\documentclass[reqno,12pt]{amsart}

\usepackage[marginratio=1:1,height=600pt,width=404pt,tmargin=110pt]{geometry}

\usepackage{amsmath, amsfonts, amssymb, amscd, bezier, amstext, amsthm}
\usepackage{color}
\usepackage{graphicx}
\usepackage{fancyhdr}
\usepackage{array}
\usepackage{esint}
\usepackage{hyperref}
\usepackage{enumitem}

\newcommand{\R}{\mathbb{R}}

\newcounter{alpha}

\theoremstyle{plain}

\newtheorem{theorem}{Theorem}[section]

\newtheorem{lemma}[theorem]{Lemma}
\newtheorem{corollary}[theorem]{Corollary}
\newtheorem{proposition}[theorem]{Proposition}

\newtheorem{remark}[theorem]{Remark}
\newtheorem{definition}[theorem]{Definition}

\title[Nonlinear elliptic systems with critical exponent]{Singular solutions to Yamabe-type systems  with prescribed asymptotics}

\author[R. Caju]{Rayssa Caju}
\author[J.M do \'O]{Jo\~ao Marcos do \'O}
\author[A. Santos]{Almir Silva Santos}

\address[R. Caju]{Department of Mathematics,
	\newline\indent 
	Federal University of Para\'{\i}ba
	\newline\indent
	58051-900, Jo\~ao Pessoa-PB, Brazil}
\email{\href{mailto:rayssacaju@gmail.com}{rayssacaju@gmail.com}}

\address[J.M. do \'O]{Department of Mathematics,
	\newline\indent 
	Federal University of Para\'{\i}ba
	\newline\indent
	58051-900, Jo\~ao Pessoa-PB, Brazil}
\email{\href{mailto:jmbo@pq.cnpq.br}{jmbo@pq.cnpq.br}}

\address[A. S. Santos]{Department of Mathematics,
	\newline\indent 
	Federal University of Sergipe
	\newline\indent
	49100-000, S\~ao Cristov\~{a}o-SE, Brazil}
\email{\href{mailto:almir@mat.ufs.br}{almir@mat.ufs.br}}
\date{}

\thanks{Research supported in part by INCTmat/MCT/Brazil, CNPq (305726/2017-0), CAPES and FAPITEC/SE/Brazil.}

\numberwithin{equation}{section}

\begin{document}
	
%
%

\begin{abstract}
Our primary purpose is to study a class of strongly coupled nonlinear elliptic systems with critical growth in a compact Riemannian manifold with constant scalar curvature. Using a gluing technique and perturbation arguments, we show the existence of singular solutions asymptotic to a Fowler-type solution near the isolated singularity.
\end{abstract}

\maketitle

%
%


%
%

\section{Introduction}

%
%

Let $(M^{n}, g)$ be an $n$-dimensional, compact, smooth, Riemannian manifold without boundary and let $p\in M$ be a fixed point. 
Our primary purpose in this paper is to establish the existence of positive singular solutions for a class of strongly coupled elliptic systems involving critical growth in the sense of the Sobolev embedding in the inhomogeneous context of a compact Riemannian manifold. 
Precisely, we look for  positive singular  solutions $\mathcal{U}=(u_1,\dots,u_d) : M \setminus  \{p\} \rightarrow \mathbb{R}^d $ to the following Yamabe-type system, 
\begin{equation}\label{S}\tag{$\mathcal{S}$}
\Delta_{g}u_{i} - \sum_{j=1}^{d}A_{ij}(x)u_{j} + \frac{n(n-2)}{4}|\mathcal{U}|^{\frac{4}{n-2}}u_{i} = 0, \quad i=1,\cdots, d,
\end{equation}
where $\Delta_{g}$ is the Laplace Beltrami operator on $M$ and $|\mathcal{U}|^2=\sum_{i=1}^{d} u_i^2$ is the Euclidean norm of the map $\mathcal{U}$.
Here  $A$ is a smooth map from the manifold $M$ into the vector space of symmetric $d\times d$ real matrices $M_{d}^{s}(\R)$.

%
%

Critical vector-valued Schr\"{o}dinger equations  of the form  \eqref{S} are weakly coupled by the matrix $A$ 
and strongly coupled by the Gross–Pitaevskii type nonlinearity in \eqref{S}. 
We say that $\mathcal U$ is a {\it positive} solution of \eqref{S} if each coordinate $u_i$ is a positive function for $i=1,\ldots,d$.
A smooth solution of  \eqref{S} in  $  M  \backslash \{p\} $ is called singular if has a singularity at the point $p$ in the sense that 
$ \limsup_{x\rightarrow p} |\mathcal{U}(x)|=\infty$.   In \cite{COS}, it was proved that for any singular solution of \eqref{S}, we have 
$ \limsup_{x\rightarrow p} u_i(x)=\infty$ for $i=1,\dots,d$.

%
%

System \eqref{S} for  $d=1$  corresponds to  the classical Yamabe equation, if $A = \frac{n-2}{4(n-1)}R_{g}$, where $R_{g}$ is the scalar curvature of $M$, which is related to the conformal geometry of the manifold.  Precisely, the Yamabe equation naturally appears in the transformation law for the scalar curvature of two metrics in the same conformal class, and it was the first equation with a critical Sobolev growth that was studied in the literature.  Still on the same subject in celebrated works \cite{KazdanWarner,KazdanWarner-JDG} J.~Kazdan and F.~Warner considered the problem
\begin{equation}\label{eq033}
\Delta_{g}u - a(x)u + f(x)u^{\frac{n+2}{n-2}} = 0,
\end{equation}
where $a$ and $f$ are arbitrary functions. 
Motivated by geometric applications and  its own mathematical interest, this type of equations have been extensively investigated in the past few decades under several assumptions on the potential $a(x)$ and in the function $f(x)$. See \cite{Hebeyfrench,HebeyVaugon,AdirmuthiYadava} and references therein. 

%
%

The Yamabe-type system \eqref{S} was initially studied by the works of O.~Druet and E.~Hebey \cite{DH} and O.~Druet, E.~Hebey and J. V\'etois \cite{DHV}. They proved stability properties of nonnegative solutions of \eqref{S} on a compact manifold $M$, under the natural assumption that the potential is related with the geometric threshold potential of the conformal Laplacian. Namely, let
$$A_n:=A-\frac{n-2}{4(n-1)}R_g \mathrm{I}_d,$$
where $R_g$ is the scalar curvature of the metric $g$ and $(\mathrm{I}_d)_{ij}=\delta_{ij}$ is the identity $d\times d$ matrix. In \cite{DH} the authors used the assumption that $A_n$ should not possess stable subspaces with an orthonormal basis consisting of isotropic nonnegative vectors, while in \cite{DHV} they assumed that $A_n<0$ in the sense  of bilinear forms. A similar critical elliptic system in potential form was studied by E. Hebey in \cite{Hebey_2006}.

In \cite{Hebey}, E. Hebey studied nonsingular solution of system \eqref{S}. Assuming some additional conditions, he proved that there exists $\varepsilon=\varepsilon(M,g,d)>0$ such that if $\|A\|<\varepsilon$, then the unique nonsingular solutions of  system \eqref{S} are the constant functions. Also he pointed out that, if $(\mathbb S^n,g_0)$ is the unit sphere with scalar curvature $R_{g_0}=n(n-1)$, by a result of  B.~Gidas and J.~Spruck \cite{Gidas_Spruck_1981} (see also \cite{Brezis_Li_2006}) we have
$$\varepsilon(\mathbb S^n,g_0,d)\leq\frac{n-2}{4(n-1)}R_{g_0}$$
for all $d$, with equality when $d=1$. In view of these works it seems very natural to impose a condition on the decay of $A_n$ close to the singularity. In fact, we will consider the assumption
\begin{flalign}\tag{$H$}\label{H1}
A_n= O(|x|^{\frac{n-3}{2}}) \quad \mbox{as} \quad x \rightarrow p.
\end{flalign}

Solutions of the equation \eqref{eq033} that develop singularities in a certain subset $F\subset M$ of the manifold were extensively studied. It is well known that the existence of such solutions is related with the size of the singular set $F$ and the sign of the scalar curvature. It turns out that the problem of finding singular solutions with isolated singularities revealed more challenging than the other cases. We refer the reader to the papers \cite{Byde, Han_Li_Teixeira_2009,KMPS,Loewner_Nirenberg_1974,MP,MPU,almir} and the references contained therein. 

Since P.-L.~Lions \cite{Lions_1982} raised the question of how far his results on the existence of positive solutions of semilinear elliptic equations may be generalized to systems of the type $-\Delta u_i=f(u_1,\ldots,u_m)$, this subject has become an active research area in recent years. In \cite{GMMY1997}, M. Garc\'ia-Huidobro et al. studied positive radial solutions to the elliptic system 
\begin{align}
\left\{\begin{array}{l}
-\Delta u=|v|^{p-1} v\\
-\Delta v=|u|^{q-1} u 
\end{array}\right.\label{ar}
\end{align}
in the punctured unit ball $\Omega^{*}=\left\{x \in \mathbf{R}^{N}: 0<|x|<1\right\}$ with $ N \geq 3$, where $p$ and $q$ are positive real numbers. They proved existence of solutions with Dirichlet boundary condition to \eqref{ar} using a Schauder fixed point argument, when the exponents $(p,q)$ lie on the open region enclosed by the critical hyperbola.



Recently, motivated by the celebrated work due to L. A. Caffarelli, B. Gidas and J. Spruck \cite{CGS} and its geometric version studied by F. C. ~Marques in  \cite{marques} and J. Xiong, L. Zhang in \cite{XiongZhang}, 
we have analysed qualitative properties of solutions with isolated singularities for system \eqref{S} in \cite{COS}. For recent developments in the study of qualitative properties of coupled elliptic systems we refer to the reader the works of \cite{Ghergu_Kim_Shahgholian_2020}, \cite{Li_Bao_2020}, \cite{CL}.

Mainly motivated by our first work  \cite{COS}, where we studied the asymptotic behavior of local solutions for coupled critical elliptic systems near an isolated singularity, it is natural to expect that the solutions of \eqref{S} are asymptotic to a Fowler-type solution. To be more specific, let us first recall the issue of deriving asymptotics for singular solutions to the following equation
\begin{equation}\label{fowler-1a}
\Delta u + \frac{n(n-2)}{4}u^{\frac{n+2}{n-2}} = 0.
\end{equation}
In \cite{CGS}, L. A. Caffarelli, B. Gidas and J. Spruck proved that any solution of  \eqref{fowler-1a} in $\R^{n}\backslash\{0\}$ with a nonremovable singularity is radial. These solutions are known in the literature as \emph{Fowler} or \emph{Delaunay-type} solutions, and they are local models to singular solutions of \eqref{fowler-1a} in the punctured unit ball $B^{n}_{1}(0)\backslash \{0\}$ (see also \cite{KMPS}). More precisely, if $u$ is a positive solution of  \eqref{fowler-1a} in the punctured unit ball then either $u$ is regular or the origin is a nonremovable singularity of $u$ and there is a Fowler solution $u_0$ such that
\begin{equation*}
u(x)=(1+o(1))u_0(x) \quad \text{as}\quad x \rightarrow 0.
\end{equation*}

In this direction, we have proved that any singular solution of the limit system
\begin{equation}\label{LSa}
\Delta u_{i} + \frac{n(n-2)}{4}| \mathcal{U} |^{\frac{4}{n-2} }u_{i}=0 \quad \mbox{in} \quad \R^{n}\setminus \{0\}, \;\;\; \mbox{for } i=1,\dots,d.
\end{equation}
is radially symmetric, see \cite[Theorem 1.2]{COS}. Moreover, we were able to obtain that the unique $C^2$ nonnegative singular solutions of the limit system \eqref{LSa} are \emph{Fowler-type solutions}, namely $\mathcal U=u_0\Lambda$ where $u_0$ is a Fowler solution and $\Lambda \in \mathbb{S}^{d-1}_{+} = \{x \in \mathbb{S}^{d-1} : x_{i} \geq 0\}$. Using this classification result we proved the Main Theorem of the work \cite{COS} which says that, under certain conditions on the potential $A$, if $3\leq n\leq 5$, then every solution of the system \eqref{S} with nonremovable isolated singularity is asymptotic to some Fowler-type solution.

Our main theorem generalizes the correspondent result of the third author \cite{almir} on the scalar case. In this latter case, since the Yamabe equation always has a solution in a closed Riemannian manifold (see \cite{Lee_Parker}), then it was natural to work with the metric with constant scalar curvature. This implies that the Yamabe equation has a constant function as a solution. Since in the system context, we do not have a trivial solution, we will assume that some $\Lambda\in\mathbb S_+
^{d-1}$ is a solution of the system \eqref{S}. The technique that we will apply in this work is known as {\it gluing method}. Usually, for this method to work it is necessary some kind of non-degeneracy assumption to assure the invertibility of some operator. In this paper, we will assume the following definition. 

	\begin{definition}\label{def001}
		A metric $g$ is nondegenerate at $\Lambda\in\mathbb S^{d-1}_+$ if the operator $\mathcal L_{g}: C^{2,\alpha}(M)^d\rightarrow C^{0,\alpha}(M)^d$, given by 
		\begin{equation}\label{eq028}
		\mathcal L_{g}^{i}(\mathcal{U}) = \Delta_{g}u_{i} - \sum_{j=1}^{d}A_{ij}u_j+n\langle \Lambda,\mathcal U\rangle \Lambda_{i} + \frac{n(n-2)}{4}u_{i},
		\end{equation}
		is surjective for some $\alpha\in(0,1)$. Here $\mathcal L_g=(\mathcal L_g^1,\ldots,\mathcal L_g^d)$ is the linearization of the system \eqref{S}, see Section \ref{sec001}, and $C^{k,\alpha}(M)$ is the H\"older space on $M$.
	\end{definition}

Now we are ready to state our main result of this work.
\begin{theorem}\label{teo004}
	Let $(M,g_0)$ be a closed Riemannian manifold with dimension $n\geq 3$ with constant scalar curvature $n(n-1)$. Suppose that
	\begin{enumerate}
	    \item the metric is nondegenerate at some $\Lambda\in\mathbb S^{d-1}_+$ such that $\Lambda$ is a trivial solution of the system \eqref{S}.
	    \item the potential $A$ satisfies the hypothese \eqref{H1}.
	    \item for dimension $n\geq 6$, the Weyl tensor vanishes up to order $\left[\frac{n-6}{2}\right]$.
	\end{enumerate}
		Then there exists a constant $\varepsilon_0>0$ and a one-parameter family of smooth functions $\mathcal V_\varepsilon=(v_{1\varepsilon},\ldots,v_{d\varepsilon})$ on $M\backslash\{p\}$ defined for all $\varepsilon\in(0,\varepsilon_0)$ such that each $\mathcal V_\varepsilon$ is a solution to the system \eqref{S} and is asymptotically to some Fowler-type solution  $\mathcal U_{\varepsilon,R,a}=u_{\varepsilon,R,a}\Lambda $ around the singular point $\{p\}$. Besides, on compact sets of $M\backslash\{p\}$ we have that $\mathcal V_\varepsilon$ converges to $\Lambda$ as $\varepsilon\rightarrow 0$.
\end{theorem}

The assumption for higher dimensions is expected, since in the scalar case \cite{almir} the problem is completely solved without additional assumptions only when the dimension is between 3 and 5. For higher dimensions it was necessary to assume that the Weyl tensor vanishes to sufficiently high order at the singular point. This condition comes up naturally due to the {\it Weyl Vanishing Conjecture}, which is a geometric conjecture due to R. Schoen \cite{Schoen_1991} for dimension greater than or equal to 6. This conjecture was proved to be true for dimension $6\leq n\leq 24$ (see \cite{Khuri_2009,Li_Zhang_2005,Li_zhang_2007,Marques_2005}) and false for $n\geq 25$ (see \cite{Marques_2009}).  In the system case the method works without any extra assumption on the metric only for dimension $3\leq n\leq 5$. For higher dimensions, we have the following result.

To prove Theorem \ref{teo004} we will use a \emph{gluing method} which consists basically of three main steps. In the first one we deal with the problem locally, that is, in a small punctured neighborhood around the point $p$.  Thus, using a fixed point argument and the hypotheses \eqref{H1}, we perturb it to obtain a family of solutions with prescribed boundary. In order to do that we perform a conformal change of the metric to obtain the right estimates that we need, see Section \ref{sec002}. In the second step, since by hypotheses $\Lambda$ is a trivial solution of the system \eqref{S}, we use it as an approximate solution to find a family of solution in the complement of some neighborhood around the point $p$ with prescribed boundary. Finally, in the third step we use properties of the Fowler solution $u_0$ and the estimates obtained in the previous steps to show that we can find suitable parameters such that the solutions in the neighborhood and in its complement coincides up to order 1. Thus using elliptic regularity we conclude that the solution is actually smooth. 

This gluing technique is similar to that employed by many authors in the literature, for instance, see the works \cite{Ao,Byde,jleli,MP,almir} and the reference contained therein. Up to our best knowledge, this is the first time that this technique is applied to obtain existence of solution to a system on a manifold. We expect that this technique can be applied to obtain existence result to other classes of systems.



It is important to mention here that the linear analysis (Section \ref{sec007}) represents the heart of our work. It is by using the right inverse obtained in the linear analysis that we can perturb the approximate solution to an exact solution through a fixed point argument. The linear program is divided in two parts, the local analysis on the punctured geodesic ball and on the complement of the geodesic ball. The first part (Section \ref{sec008}) consist in to find a right inverse to the linearization of the operator in a punctured ball in $\mathbb R^n$ about a Fowler-type solution, while the second part (Section \ref{sec006}) consist in to find a right inverse to the linearization in the complement of a geodesic ball of the manifold. The analysis in the first part is more delicate and we follow the ideas by R. Mazzeo and F. Pacard \cite{MP} where a separation of variables technique were used. We extended their result to the system context. For this reason we deal separately with the space orthogonal to the constant and coordinates functions, where we get coercivity and the solution to the system are guaranteed, and with the space generated by these functions, where we get a solution by ODE theory. Finally we obtain a uniform {\it a priori} estimate for these solutions to ensure that the right inverse is bounded independently of the parameters. 


We briefly describe the outline of the paper. For most of the paper we will consider the case $3\leq n\leq 5$. In the Section \ref{sec004} we review the main facts about the Fowler-type solution, including the classification result about the limit system. Then we define the operator and the functions spaces that we are going to work with. In Section \ref{sec007} we analyzes the linearization of the operator. We find a right inverse to the linearized operator in the punctured ball in $\mathbb R^n$ and also a right inverse to the linearized operator about $\Lambda$ in the complement of a geodesic ball. Using the previous right inverse, in Section \ref{sec009} we use a fixed point argument to find a family of solution to the system \eqref{S} in a punctured geodesic ball and in its complement in the manifold. In Section \ref{sec005} we use the results obtained in the previous sections to prove the main result of this paper for dimension $3\leq n\leq 5$. Finally, in the briefly Section \ref{sec011} we finish our work explain how modify the previous arguments to prove the Theorem \ref{teo004} for high dimensions.

\section{Preliminaries}\label{sec004}

\subsection{Fowler-type solutions}

The notable work of Caffarelli, Gidas and Spruck \cite{CGS} stated that the \emph{Delaunay} or \emph{Fowler solutions} are the suitable models to arbitrary solutions of the singular Yamabe equation with isolated singularities. Remember that a Fowler or Delaunay solution is a positive function satisfying the equation
\begin{equation}\label{fowler}
\mathcal{N}(u) = \Delta u + \frac{n(n-2)}{4}u^{\frac{n+2}{n-2}} = 0 \quad \mbox{in}\quad \R^{n}\setminus \{0\}.
\end{equation}
with a nonremovable singularity at the origin.

In \cite{MP} it was proved that for each $\varepsilon\in(0,((n-2)/n))^{(n-2)/4}$ there exists a positive radial solution $u_\varepsilon$ of \eqref{fowler} given by $u_\varepsilon(x)=|x|^{\frac{2-n}{2}}v(-\log |x|)$, where the function $v$ is periodic and $\varepsilon=\min v(t)$. To see details about the Fowler solutions and its properties we refer the reader to \cite{KMPS,MPU,almir}.

If $R \in \R^{+}$ and $a \in \R^{n}$, there are some important variations of the Fowler solutions that still solves the equation \eqref{fowler}, possibly in a small punctured ball, given by 
\begin{equation}\label{fowf}
u_{\varepsilon, R, a}(x) := |x -a|x|^2|^{\frac{2-n}{2}}v_{\varepsilon}(-2\log x + \log |x - a|x|^2| + \log R).
\end{equation}
If $a=0$ we denote simply as $u_{\varepsilon,R}$ and if $R=1$ by $u_\varepsilon$. In the next result we will summarize some properties of \eqref{fowf} that we will be useful throughout this work.
\begin{proposition}\cite{MP,almir}\label{propo002}
	Given $u_{\varepsilon,R,a}$ as in \eqref{fowf} we have the following
	\begin{enumerate}
		\item 	For $x\in\mathbb{R}^n$ with $|x|\leq 1$, we have
		%
		%
		$$ u_{\varepsilon,R}(x) = \displaystyle\frac{\varepsilon}{2} \left(R^{\frac{2-n}{2}}+ R^{\frac{n-2}{2}}|x|^{2-n}\right)+ O''(R^{\frac{n+2}{2}}\varepsilon^{\frac{n+2}{n-2}}|x|^{-n}),$$
		$$|x|\partial_ru_{\varepsilon,R}(x)=\frac{2-n}{2}\varepsilon R^{\frac{n-2}{2}} |x|^{2-n}+ O'(R^{\frac{n+2}{2}}\varepsilon^{\frac{n+2}{n-2}}|x|^{-n})$$
		and
		$$|x|^2\partial_r^2u_{\varepsilon,R}(x)=\frac{(n-2)^2}{2}\varepsilon R^{\frac{n-2}{2}} |x|^{2-n}+ O(R^{\frac{n+2}{2}}\varepsilon^{\frac{n+2}{n-2}}|x|^{-n}).$$
		\item There exists a constant $r_0\in(0,1)$, such that for $x,a\in\mathbb{R}^n$ with $|x|\leq 1$, $|a||x|<r_0$, we have
		%
		%
		\begin{equation*}
		u_{\varepsilon,R,a}(x) =  u_{\varepsilon,R}(x)+((n-2)u_{\varepsilon,R}(x)+ |x| \partial_ru_{\varepsilon,R}(x))a\cdot x+O''(|a|^2|x|^{\frac{6-n}{2}}).
		\end{equation*}
		Besides if $R\leq |x|$, we have
		\begin{equation*}
		\begin{array}{rcl}
		u_{\varepsilon,R,a}(x) & = & u_{\varepsilon,R}(x)+((n-2)u_{\varepsilon,R}(x)+ |x| \partial_ru_{\varepsilon,R}(x))a\cdot x\\
		\\
		& + & O''(|a|^2\varepsilon R^{\frac{2-n}{2}}|x|^2).
		\end{array}
		\end{equation*}
	\end{enumerate}
\end{proposition}

Similarly to the results obtained by Caffarelli, Gidas and Spruck in \cite{CGS}, in \cite{COS} we studied the solutions of the following system
\begin{equation}\label{LS}
\Delta u_{i} + \frac{n(n-2)}{4}| \mathcal{U} |^{\frac{4}{n-2} }u_{i}=0 \quad \mbox{in} \quad \R^{n}\setminus \{0\}.
\end{equation}
Positive solutions of this problem with a nonremovable singularity at the origin, meaning that $ \displaystyle\lim_{x\rightarrow 0} |\mathcal{U}(x)| = +\infty$, are called \emph{Fowler-type solutions}.

We proved in \cite{COS} that positive solutions of \eqref{LS} are radially symmetric. Moreover, we obtained a fully characterization of the singular solutions of \eqref{LS} in terms of the Fowler solutions described initially. 

\begin{theorem}\cite[Theorem 1.2]{COS}\label{class1} Let $\mathcal{U} \in C^2(\R^n \backslash \{0\})$ be a nonnegative singular solution of \eqref{LS}.
	Then there exists  $\Lambda \in \mathbb{S}^{d-1}_{+} = \{x \in \mathbb{S}^{d-1} : x_{i} \geq 0\}$ and a Fowler solution $u_\varepsilon$ such that 
	$\mathcal{U} = u_{\varepsilon}\Lambda.$
\end{theorem}

In light of the aforementioned strategy for the singular Yamabe problem and with this recent classification result it is natural then seek solutions to \eqref{S}, that are asymptotic to $\mathcal{U}_{\varepsilon,R,a} = u_{\varepsilon,R,a}\Lambda$, where $u_{\varepsilon,R,a}$ is given by \eqref{fowf}.

\subsection{The operator}\label{sec001}

Let $(M^{n},g)$ be a compact Riemannian manifold and let $p \in M$ be a fixed point. In order to find a positive singular solution to the system \eqref{S} in $M\backslash\{p\}$, define the quasilinear map $H_{g}(\mathcal{U}) = (H_{g}^{1}(\mathcal{U}),\ldots, H_{g}^{d}(\mathcal{U}))$, where $\mathcal U=(u_1,\ldots,u_d)$, with components given by
$$H_{g}^{i}(\mathcal{U}) = \Delta_{g}u_{i} - \sum_{j=1}^{d}A_{ij}u_j + \frac{n(n-2)}{4}|\mathcal{U}|^{\frac{4}{n-2}}u_{i}.$$

Using that the Conformal Laplacian $L_g=\Delta_g-\frac{n-2}{4(n-1)}R_g$ obeys the following relation
concerning conformal changes of the metric,
$$L_{v^{\frac{4}{n-2}}g}(u)=v^{-\frac{n+2}{n-2}}L_g(vu),$$
we obtain that
\begin{equation}\label{eq027}
    \begin{array}{rcl}
     H^i_{v^{\frac{4}{n-2}}g}(\mathcal U) & = &\displaystyle v^{-\frac{n+2}{n-2}}\left(L_g(vu_i)+ \frac{n(n-2)}{4}|v\mathcal U|^{\frac{4}{n-2}}vu_i\right.\\
     \\
     & & \displaystyle\left.-v^{\frac{4}{n-2}}\sum_j\left(A_{ij}-\frac{n-2}{4(n-1)}R_{v^{\frac{4}{n-2}}g}\delta_{ij}\right){vu_i}\right).
\end{array}
\end{equation}

Although a Fowler-type solution $\mathcal U_0=u_0\Lambda$ is not a solution to the equation $H_g=0$, we will use it as a approximate solution. Here $\Lambda\in \mathbb S_+^1$ is fixed. Thus,  we will seek a function $\mathcal{U}$ such that
$$
\left\{
\begin{array}{lcl}
H_{g}(\mathcal{U}_0 + \mathcal{U}) = 0 & \mbox{ in }  & M\backslash\{p\}\\
\\
(\mathcal{U}+\mathcal{U}_0)(x)\rightarrow \infty & \mbox{ as } & x\rightarrow p.
\end{array}
\right.$$

This is done by considering the linearization of the quasilinear map $H_{g}$ about $\mathcal{U}_0=u_0\Lambda$, which is defined by $\mathcal{L}_{g}(\mathcal{U}) = (\mathcal{L}_{g}^{1}(\mathcal{U}),\ldots, \mathcal{L}_{g}^{d}(\mathcal{U}))$ where
\begin{equation}\label{li}
\begin{array}{lcl}
\mathcal{L}_{g}^{i}(\mathcal{U}) &= &\displaystyle \frac{\partial}{\partial t}\bigg|_{t=0}H_{g}^{i}(\mathcal{U}_{0} + t\mathcal{U}) \\
&= & \displaystyle\Delta_{g}u_i - \sum_{j=1}^{d}A_{ij}u_j + nu_{0}^{\frac{4}{n-2}}\langle \Lambda,\mathcal{U}\rangle \Lambda_i + \frac{n(n-2)}{4}u_{0}^{\frac{4}{n-2}}u_i,
\end{array}
\end{equation}
where $\Lambda=(\Lambda_1,\ldots,\Lambda_d)\in\mathbb S^{d+1}_+$. Expanding $H_{g}$ about a Fowler-type solution $\mathcal{U}_0$, we obtain
\begin{equation}\label{eq057}
H_{g}(\mathcal{U}_{0} + \mathcal{U}) = H_{g}(\mathcal{U}_{0}) + \mathcal{L}_{g}(\mathcal{U}) + Q(\mathcal{U}),
\end{equation}
where the nonlinear remainder term
 $Q(\mathcal{U}) = (Q^{1}(\mathcal{U}),\ldots, Q^{d}(\mathcal{U}))$ is given by
 \begin{equation}\label{resto}
\begin{array}{rcl}
Q^{i}(\mathcal{U}) & = &\displaystyle\frac{n(n-2)}{4}|\mathcal{U}_{0} + \mathcal{U}|^{\frac{4}{n-2}}(u_0\Lambda_i+u_i) - \frac{n(n-2)}{4}u_0^{\frac{n+2}{n-2}}\Lambda_i\\
\\
& &\displaystyle- nu_0^{\frac{4}{n-2}}\langle \mathcal{U}, \Lambda \rangle \Lambda_i - \frac{n(n-2)}{4}u_{0}^{\frac{4}{n-2}}u_i\\
\\
& = & \displaystyle n\int_0^1 \left(|\mathcal U_0+t\mathcal U|^{\frac{8-2n}{n-2}}\langle \mathcal U_0+t\mathcal U,\mathcal U\rangle (u_0\Lambda_i+tu_i)-u_0^{\frac{4}{n-2}}\langle \mathcal U,\Lambda\rangle\Lambda_i \right.\\
\\
& & +\displaystyle \frac{n(n-2)}{4}\left.\left(|\mathcal U_0+t\mathcal U|^{\frac{4}{n-2}}u_i-u_0^{\frac{4}{n-2}}u_i\right) \right)dt.
\end{array}
\end{equation}

\subsection{Function spaces}
In this section we will quickly review some definitions of weighted H\"{o}lder spaces introduced by Mazzeo and Pacard in \cite{MP} and we will fix some important notations. For details we also refer the reader to \cite{almir}.

We emphasize that the advantage of using weighted H\"{o}lder spaces lies in the fact that by choosing well the weight, it is possible to show that the linearized operator has a right inverse when we consider some additional boundary restrictions. 

\begin{definition}\label{d1} For each $k \in \mathbb{N}$, $r > 0$, $0 < \alpha < 1$, $\sigma \in (0,r/2)$ and $u \in C^{k}(B_{r}(0)\backslash \{0\})$, set
	\begin{equation*}
	\|u\|_{(k,\alpha),[\sigma, 2\sigma]} := \sup_{|x| \in [\sigma, 2\sigma]}\left(\sum_{j=0}^{k}\sigma^{j}|\nabla^{j}u(x) |\right) + \sigma^{k+\alpha}\sup_{|x|, |y| \in [\sigma, 2\sigma]}\frac{|\nabla^{k}u(x) - \nabla^{k}u(y)|}{|x - y|^{\alpha}}.
	\end{equation*}
	Thus, for any $\mu \in \R$, the space $C^{k,\alpha}_{\mu}(B_r(0)\backslash\{0\})$ is the collection of functions $u$ that are locally in $C^{k,\alpha}(B_r(0)\backslash \{0\})$ and for which the following norm is finite,
	\begin{equation*}
	\|u\|_{(k,\alpha),\mu,r} := \sup_{0 < \sigma \leq \frac{r}{2}} \sigma^{-\mu}\|u\|_{(k,\alpha), [\sigma, 2\sigma]}.
	\end{equation*}
\end{definition}

\begin{definition} Let $k \in \mathbb{N}$, $\alpha\in(0,1)$ and $r > 0$. The space $C^{k,\alpha}(\mathbb{S}^{n-1}_{r})$ is the collection of functions $\phi \in C^{k}(\mathbb{S}^{n-1}_{r})$ for which the norm 
$$	\|\phi\|_{(k,\alpha),r} := \|\phi(r\cdot)\|_{C^{k,\alpha}(\mathbb{S}^{n-1})}.$$
 is finite. Here $\mathbb S_r^{n-1}:=\{x\in\mathbb R^n;|x|=r\}$.
\end{definition}

Next, consider a compact Riemannian manifold $(M,g)$ and $\Psi: B_{r_{1}}(0)\rightarrow M$ some coordinate system on $M$ centered at some point $p \in M$, where $B_{r_{1}}(0) \subset \R^{n}$ is the ball of radius $r_{1}>0$ fixed. For $0 < r < s \leq r_{1}$, define $M_{r}:= M\backslash \Psi(B_{r}(0))$
and $\Omega_{r,s} := \Psi(A_{r,s}),$
where $A_{r,s} := \{x\in \R^{n};  r \leq |x| \leq s \}$.

Define the spaces $C^{k,\alpha}_{\mu}(\Omega_{r,s})$ and $C^{k,\alpha}_{\mu}(M_{r})$ as the space of restriction of elements of $C^{k,\alpha}_{loc}(M\backslash\{p\})$ to $M_{r}$ and $\Omega_{r,s}$, such that the norms
\begin{equation*}
	\|f\|_{C^{k,\alpha}_{\mu}(\Omega_{r,s})} := \sup_{r\leq \sigma \leq \frac{s}{2}}\sigma^{-\mu}\|f\circ \Psi\|_{(k,\alpha),[\sigma,2\sigma]}
\end{equation*}
and 
\begin{equation*}
	\|h\|_{C^{k,\alpha}_{\mu}(M_{r})} := \|h\|_{C^{k,\alpha}(M_{\frac{1}{2}r_{1}})} + \|h\|_{C^{k,\alpha}_{\mu}(\Omega_{r,r_{1}})},
\end{equation*}
are finite, respectively. 

Since we are working with systems, it will be necessary to consider the product of these spaces. We will indicate the number of factors as a index separated by a semicolon. For example
\begin{equation*}
C^{2,\alpha}_{\mu;l}(B_{r}(0)\backslash \{0\}) = \underbrace{C^{2,\alpha}_{\mu}(B_{r}(0)\backslash \{0\}) \times \cdots \times C^{2,\alpha}_{\mu}(B_{r}(0)\backslash \{0\})}_{l-times} ,
\end{equation*}
whose norm of an element $\mathcal{U} = (u_{1},...,u_{l}) \in C^{2,\alpha}_{\mu;l}(B_{r}(0)\backslash \{0\})$ is given by
$$\|\mathcal{U}\|_{(k,\alpha),\mu,r;l} = \sum_{j=1}^{l}\|u_{j}\|_{(k,\alpha),\mu,r}.$$

\section{Linear Analysis}\label{sec007}

We consider the nonlinear operator \eqref{eq027} which we linearize to obtain \eqref{li}. Our main goal in this section is to find a right inverse of the linearization. First we prove that the linearization of the equation \eqref{LS} about a Fowler-type solution has right inverse in the punctured ball in $\mathbb R^n$. Then, we use the nondegeneracy assumption (Definition \ref{def001}) to find a right inverse of the linearization about $\Lambda$ in the complement of a ball in the manifold.

\subsection{Analysis in the punctured  ball in \texorpdfstring{$\mathbb R^n$}{Lg}}\label{sec008}

In this section we want to study the linearization of the system \eqref{LS} about a Fowler-type solution, which is given by $\mathcal L_{\varepsilon,R,a}(\mathcal U)=(\mathcal L_{\varepsilon,R,a}^1(\mathcal U),\ldots,\mathcal L_{\varepsilon,R,a}^d(\mathcal U))),$
where
$$\mathcal L_{\varepsilon,R,a}^i(\mathcal U)=\Delta u_i+nu_{\varepsilon,R,a}^{\frac{4}{n-2}}\langle \Lambda,\mathcal U\rangle\Lambda_i+\frac{n(n-2)}{4}u_{\varepsilon,R,a}^{\frac{4}{n-2}}u_i,$$
for all $i=1,\ldots,d$, see \eqref{li}. For simplicity we denote $\mathcal L_{\varepsilon,R,0}$ by $\mathcal L_{\varepsilon,R}$ and when $R=1$, simply by $\mathcal L_\varepsilon$.
Our goal in this section is to study the Dirichlet problem in the punctured unit ball
\begin{equation}\label{sol}
\left\{
\begin{array}{lcl}
\mathcal{L}_{\varepsilon,R}(\mathcal W) = f &\mbox{ in } & \quad B_{1}(0)\backslash\{0\} \\
\mathcal W = 0 & \mbox{ on } & \quad \partial B_{1}(0),
\end{array}
\right.
\end{equation}
where  $\mathcal W=(w_1,\ldots,w_d)$, $f=(f_1,\ldots,f_d)$, $\mathcal L_{\varepsilon,R}=(\mathcal L_{\varepsilon,R}^1,\ldots,\mathcal L_{\varepsilon,R}^d)$ and
\begin{equation}\label{eq002}
\mathcal{L}_{\varepsilon,R}^{i}(\mathcal W) =
 \displaystyle\Delta w_i+nu_{\varepsilon,R}^{\frac{4}{n-2}}\langle \Lambda,\mathcal W\rangle \Lambda_i+\frac{n(n-2)}{4}u_{\varepsilon,R}^{\frac{4}{n-2}}w_i.
\end{equation}


Replacing $f(x)$ by $R^{-2}f(x/R)$ and $\mathcal W(x)$ by $\mathcal W(x/R)$, \eqref{sol} is equivalent to 
\begin{equation}\label{sol2}
\left\{
\begin{array}{lcl}
\mathcal{L}_{\varepsilon}(\mathcal W) = f &\mbox{ in } & \quad B_{1/R}(0)\backslash\{0\} \\
\mathcal W = 0 & \mbox{ on } & \quad \partial B_{1/R}(0).
\end{array}
\right.
\end{equation}

Due to the geometry of our domain, it is natural to approach our problem by means of a classical
separation of variables. To do that we consider the spectral decomposition of the operator $\Delta_{\mathbb{S}^{n-1}}$. Let $\phi_k$ be the eigenfunction of the Laplace-Beltrami operator on $\mathbb S^{n-1}$, that is, $\phi_k$ satisfies the identitie $\Delta_{\mathbb S^{n-1}}\phi_k+\lambda_k\phi_k=0$, with $k\in\mathbb N$. We recall that the spectrum of $\Delta_{\mathbb S^{n-1}}$ is given by $\{k(n-2+k);k\in N\}$. In particular, the first nonzero eigenvalue is $n-1$, with multiplicity $n$, whereas for $j\geq n+1$ we have $\lambda_j\geq 2n$. Besides, we consider the following eigenfunction decomposition of $f=(f_1,\ldots,f_d)$ and $\mathcal W=(w_1,\ldots,w_d)$
\begin{align*}
w_{i}(x) = |x|^{\frac{2-n}{2}}\sum_{j=0}^{\infty}w_{ij}(-\log|x|)\phi_{j}(\theta), \\
f_{i}(x) = |x|^{-\frac{n+2}{2}}\sum_{j=0}^{\infty}f_{ij}(-\log|x|)\phi_j(\theta).
\end{align*}
Then $\mathcal{W}_{j} = (w_{1j},\ldots, w_{dj})$ solves 
\begin{equation*}
\left\{
\begin{array}{rcl}
\mathcal L_{\varepsilon,j}^{i}(\mathcal{W}_{j}) & := & \displaystyle \frac{d^{2}w_{ij}}{dt^{2}} - \frac{(n-2)^2}{4}w_{ij} - \lambda_{j}w_{ij} + nv_{\varepsilon}^{\frac{4}{n-2}}\langle\mathcal{W}_{j},\Lambda\rangle\Lambda_{i} \\
\\
& & \qquad \qquad \displaystyle+ \frac{n(n-2)}{4}v_{\varepsilon}^{\frac{4}{n-2}}w_{ij}=f_{ij} \\
w_{ij}(\log R) & = & 0,
\end{array}
\right.
\end{equation*}
where $\mathcal L_{\varepsilon,j}^{i}$ is the restriction to the space generated by $\phi_j$ of the operator 
\begin{equation}\label{eq015}
\mathcal L_{\varepsilon}^{i}(\mathcal{W}) = \frac{d^{2}w_{i}}{dt^{2}} + \Delta_{\mathbb{S}^{n-1}} w_{i} - \frac{(n-2)^2}{4}w_{i} + nv_{\varepsilon}^{\frac{4}{n-2}}\langle \mathcal{W},\Lambda\rangle\Lambda_{i} + \frac{n(n-2)}{4}v_{\varepsilon}^{\frac{4}{n-2}}w_{i},
\end{equation}
which is the problem \eqref{sol2} transformed from $\mathbb R^n\backslash\{0\}$ to $\mathbb R\times\mathbb S^{n-1}$.

In the spirit of \cite{MP}, it is convenient to treat separately the {\it high frequencies}, $\lambda_j\geq 2n$ for $j\geq n+1$, and the {\it low frequencies}, namely $j=0,\ldots, n$.

\subsubsection{High frequencies: $j\geq n+1$.}

We consider the projection of $\mathcal W$ and $f$ along the high frequencies
\begin{align*}
\overline w_{i}(t,\theta) =\sum_{j=n+1}^{\infty}w_{ij}(t)\phi_{j}(\theta), \\
\overline f_{i}(t,\theta) = \sum_{j=n+1}^{\infty}f_{ij}(t)\phi_j(\theta).
\end{align*}

We must solve
\begin{equation}\label{eq010}
\left\{
\begin{array}{lc}
\mathcal L_{\varepsilon}^{i}( \overline{\mathcal{W}} )= \overline{f}_{i} \quad \mbox{ in } \quad (\log R, +\infty)\times \mathbb{S}^{n-1}, \\
\overline{w}_{i}(\log R,\theta) = 0,
\end{array}
\right.
\end{equation}
where $\overline{\mathcal W}=(\overline w_1,\ldots,\overline w_d)$. To solve \eqref{eq010}, first we study the following family of Dirichlet boundary value problem indexed by the parameter $T\in(\log R,\infty)$,
\begin{equation}\label{eq011}
\left\{
\begin{array}{lc}
\mathcal  L_{\varepsilon}^{i}(\overline{\mathcal{W}}_{T} )= \overline{f}_{i} \quad \mbox{ in } \quad (\log R, T)\times \mathbb{S}^{n-1}, \\
\overline{w}_{iT}(\log R,\theta) =\overline{w}_{iT}(T,\theta) = 0,
\end{array}
\right.
\end{equation}
for $i=1,\ldots,d$. To find a solution of $\eqref{eq011}$, we notice that the problem has a variational structure. Indeed, it is easy to see that critical points of the Euler-Lagrange
functional
\begin{align*}
\mathcal{F}_{T}(\mathcal W) = \int_{\log R}^{T}\int_{\mathbb{S}^{n-1}}\left(|\dot{\mathcal{W}}|^{2} + \frac{(n-2)^2}{4}|\mathcal{W}|^{2} + |\nabla_{\mathbb{S}^{n-1}}\mathcal{W} |^{2} - nv_{\varepsilon}^{\frac{4}{n-2}}\langle\mathcal{W},\Lambda\rangle^{2} \right. \\ \left.- \frac{n(n-2)}{4}v_{\varepsilon}^{\frac{4}{n-2}}|\mathcal{W}|^{2} + \langle \overline{f}, \mathcal{W}\rangle\right)dtd\theta,
\end{align*}
are weak solutions of \eqref{eq011}. Since $\lambda_{j}\geq \lambda_{n+1} = 2n$ and $0<v_\epsilon<1$, we estimate
\begin{align*}
\mathcal{F}_{T}(\mathcal W) \geq \int_{\log R}^{T}\int_{\mathbb{S}^{n-1}}\left(|\dot{\mathcal{W}}|^{2} + \frac{n+2}{2}|\mathcal{W}|^{2} + \langle \overline{f}, \mathcal{W}\rangle\right)dtd\theta.		
\end{align*}
This implies that the functional $\mathcal F_T$ is coercive on
$$H_0^1(D_R^T)_{;d}^\perp:=\left\{u=(u_1,\ldots,u_d); u_i\in H_0^1(D_R^T), \;\int_{\mathbb S^{n-1}}u_i(\cdot,\theta)\phi_j(\theta)d\theta=0\right.$$
$$\left.\mbox{ for all }j=0,\ldots,n\mbox{ and } i=1,\ldots,d \right\},$$
where $D_R^T=(\log R,T)\times\mathbb S^{n-1}$. Hence it is bounded from below. Furthermore, we can check that the functional $\mathcal F_T$ is weakly lower semicontinuous on $H_0^1(D_R^T)^\perp_{;d}$. Thus, we infer the existence of a minimizer $\mathcal W_T$ of $\mathcal F_T$, which provides a (weak) solution of \eqref{eq011}. The standard elliptic theory yields the expected regularity issues for $\mathcal W_T$ in terms of the regularity of $\overline f=(\overline f_1,\ldots,\overline f_d)$.

Before we prove the main result of this section let us prove the following lemma.

\begin{lemma}\label{lema2} Consider a unit vector $\Lambda=(\Lambda_1,\ldots,\Lambda_d)$. For every $\eta > 0$, there exists $\varepsilon_{0} > 0$ such that, for all $j \geq n+1$ and for all $\varepsilon \in (0,\varepsilon_{0}]$, every solution $\mathcal W=(w_1,\ldots,w_d)$ of the system
	\begin{equation}\label{sl2}
	\frac{d^{2}w_{i}}{dt^{2}} - \left(\frac{(n-2)^2}{4}+\lambda_{j}- \frac{n(n-2)}{4}v_{\varepsilon}^{\frac{4}{n-2}}\right)w_{i} + n\langle \mathcal{W},\Lambda\rangle \Lambda_{i}  v_{\varepsilon}^{\frac{4}{n-2}}= 0,
	\end{equation}
	i=1,\ldots,d, either decays to zero faster than $e^{-(((n+2)^2/4) - 4\eta)^{1/2}t}$ at $+\infty$ (respectively, $e^{(((n+2)^2/4) - 4\eta)^{1/2}t}$ at $-\infty$) or blows up faster than $e^{(((n+2)^2/4) - 4\eta)^{1/2}t}$ at $+\infty$ (respectively, $e^{-(((n+2)^2/4) - 4\eta)^{1/2}t}$ at $-\infty$).
\end{lemma}
\begin{proof}
	Consider the equation
	\begin{equation}\label{eq1}
	\frac{d^{2}w}{dt^{2}} - \left(\frac{(n-2)^{2}}{4} + \lambda_{j}  - \frac{n(n+2)}{4}v_{\varepsilon}^{\frac{4}{n-2}}\right)w =0.
	\end{equation}
	
	We know that the space of all solutions has dimension two. Moreover any solution of \eqref{eq1} times the unit vector $\Lambda$ is a solution of $\eqref{sl2}$. Similarly, if we consider the equation
	\begin{equation}\label{eq2}
	\frac{d^{2}w}{dt^{2}} - \left(\frac{(n-2)^{2}}{4}+ \lambda_{j}  - \frac{n(n-2)}{4}v_{\varepsilon}^{\frac{4}{n-2}}\right)w =0 ,
	\end{equation}
	we have that the dimension of the space of all solutions has dimension two and any solution $w$ of \eqref{eq2} times a unit vector $\overline{\Lambda}$ orthogonal to $\Lambda$ is a solution of the system $\eqref{sl2}$.
	Since 
	$\Lambda$ and $\overline{\Lambda}$ are orthogonal vectors, the solutions constructed above  spam the space of all solutions to the system \eqref{sl2}. Now, we note that the Lemma 2 in \cite{MP} is still true if the ODE \eqref{eq1} is replaced by \eqref{eq2}. From this, the lemma follows.
\end{proof}

To the next result let us define
\begin{equation}\label{eq023}
\|f\|_{C_{\delta;d}^0(D_R^T)}:=\sup_{(t,\theta) \in D_R^T }e^{\delta t}|f(t,\theta)|,
\end{equation}
where $D_R^T=(\log R,T)\times\mathbb{S}^{n-1}$. Similarly we define $\|f\|_{C_{\delta;d}^2(D_R^T)}$.
Now, let us prove the main result of this section.

\begin{proposition}\label{prop001}
	Let $\overline{\mathcal W}_T\in C_{\delta;d}^2(D_R^T)$ be a solution of \eqref{eq011} with $\overline f\in C^0_{\delta;d}(D_R^T)$. Then, for every $\delta\in\left(-\frac{n+2}{2},\frac{n+2}{2}\right)$, there exist constants $\varepsilon_0=\varepsilon_0(\delta,n)>0$ and $C>0$, such that for every $\varepsilon\in(0,\varepsilon_0]$, $R>0$ and $T>\log R$ we have
	\begin{equation}\label{eq012}
\|\overline{\mathcal W}_T\|_{C_{\delta;d}^0(D_R^T)}\leq C\|\overline f\|_{C_{\delta;d}^0(D_R^T)}.
	\end{equation}
\end{proposition}


\begin{proof}

The argument is by contradiction. If \eqref{eq012} were not true, then there would exist a sequence  ($T_{k}$, $\varepsilon_{k}$, $R_k$, $\overline{\mathcal{W}}_{T_{k}}$, $\overline{f}_{k}$) such that 
\begin{enumerate}
	\item $\varepsilon_k\rightarrow 0$ as $k\rightarrow+\infty$.
\item $\|\overline f\|_{C_{\delta;d}^0(D_{R_k}^{T_k})} = 1.$
	\item $\displaystyle \lim_{k \rightarrow \infty} \|\overline{\mathcal W}_{T_k}\|_{C^0_{\delta;d}(D_{R_k}^{T_k})} = \infty.
	$
		\item $$\left\{
	\begin{array}{lc}
	\mathcal  L_{\varepsilon}^{i}(\overline{\mathcal{W}}_{T_k} )= \overline{f}_{i} \quad \mbox{ in } \quad (\log R_k, T_k)\times \mathbb{S}^{n-1}, \\
	\overline{w}_{iT_k}(\log R_k,\theta) = \overline{w}_{iT_k}(T_k,\theta) =0,
	\end{array}
	\right.$$
	
\end{enumerate}

Now choose $t_{k} \in (\log R_k, T_{k})$ such that 
\begin{equation*}
A_{k}:=\sup_{\theta \in \mathbb{S}^{n-1}} e^{\delta t_{k}}|\overline{\mathcal{W}}_{T_{k}}(t_{k},\theta)| = \|\overline{\mathcal W}_{T_k}\|_{C^0_{\delta;d}(D_{R_k}^{T_k})}>0
\end{equation*}
and define
$${\mathcal{W}_{k}}(t,\theta) = A_{k}^{-1}e^{\delta t_{k}}\overline{\mathcal{W}_{T_{k}}}(t+t_{k},\theta)\;\;\mbox{ and }\;\;
{f_{k}}(t,\theta) = A_{k}^{-1}e^{\delta t_{k}}\overline{f}(t+t_{k},\theta),$$
for $(t,\theta)\in (\log R_k-t_k,T_k-t_k)\times\mathbb S^{n-1}$. Then, by definition
\begin{equation}\label{eq003}
\sup_{\theta \in \mathbb{S}^{n-1}}\sup_{\log R_k -t_{k}\leq t\leq T_{k}-t_{k}} e^{\delta t}|{\mathcal{W}_{k}}(t,\theta)|=\sup_{\theta \in \mathbb{S}^{n-1}} |{\mathcal{W}_{k}}(0,\theta)| = 1,
\end{equation}
and 
\begin{equation}\label{eq005}
\sup_{(t,\theta) \in (\log R_k-t_k,T_k-t_k)\times\mathbb{S}^{n-1}} e^{\delta t}|{f_{k}}(t,\theta)|\rightarrow 0\;\;\mbox{ as }\;\;k\rightarrow+\infty.
\end{equation}
In addition,
\begin{equation}\label{eq004}
\begin{array}{l}
\displaystyle\frac{d^{2}{w}_{ik}}{dt^{2}} - \frac{(n-2)^2}{4}{w}_{ik} + \Delta_{\mathbb{S}^{n-1}}{w}_{ik}+ \\
\qquad \displaystyle+ v_{k}^{\frac{4}{n-2}}(t+t_{k})\left(n\langle {\mathcal{W}}_{k},\Lambda\rangle \Lambda_{i} + \frac{n(n-2)}{4}{w}_{ik}\right) = {f}_{ik} 
\end{array}
\end{equation}
on $[\log R_k - t_{k}, T_{k} - t_{k}]\times\mathbb S^{n-1}$, where $\mathcal W_k=(w_{1k},\ldots,w_{dk})$ and $f_k=(f_{1,k},\ldots,f_{dk})$. Up to consider a subsequence, we can assume without loss of generality that, as $k\rightarrow+\infty$, we have
$$\log R_k - t_{k}\rightarrow \tau_{1} \in \R^{-}\cup\{-\infty\}\;\;\;\mbox{ and }\;\;\;T_{k} - t_{k}\rightarrow\tau_{2} \in \R^{+}\cup \{+\infty\}.$$

Suppose that $\log R_k-t_k\rightarrow 0$ as $k\rightarrow+\infty$. Since $ \mathcal W_k(\log R_k-t_k,\theta)=0$ for every $\theta\in\mathbb S^{n-1}$ and \eqref{eq003} holds, then the quantities $|\nabla  \mathcal W_k|$ would forced to explode in a region of the type $[\log R_k-t_k,\log R_k-t_k+1]\times\mathbb S^{n-1}$ as $k\rightarrow+\infty$. On the other hand, using \eqref{eq004} and assumptions on $\mathcal W_k$ and $f_k$ we deduce that there exists a positive constant $C>0$ such that 
$$| {\mathcal W}_k|+|\Delta_{(t,\theta)}  {\mathcal W}_k|\leq Ce^{\delta t_k},$$
for every $k\in\mathbb N$, in the region $[\log R_k-t_k,\log R_k-t_k+2]\times\mathbb S^{n-1}$. Hence, by the classical Schauder estimates, the
gradients are also bounded by the same quantities in $[\log R_k-t_k,\log R_k-t_k+1]\times\mathbb S^{n-1}$, which is a contradiction. Therefore, $\tau_1\not=0$. Similarly we prove that $\tau_2\not=0$.

Now, it is well known that the $C^2$-norms of the functions $v_k(t+t_k)$ are uniformly bounded, thus, using the Arzelà-Ascoli Theorem, we deduce that there exists a function $v_\infty$ such that $v_k(t+t_k)$ converges to $v_\infty(t)$ in $C^1_{loc}(\mathbb R)$, which is also a Fowler solution. Now, using \eqref{eq003}, \eqref{eq005} and the classical interior Schauder estimates and applying the Arzelà-Ascoli Theorem, we deduce that the functions $w_{ik}$, which satisfies \eqref{eq004}, converge to some $w_{i\infty}\in C^1_{loc}((\tau_1,\tau_2)\times\mathbb S^{n-1})$, which verifies the equation
\begin{equation}\label{eq006}
\frac{d^{2}{w}_{i\infty}}{dt^{2}} - \frac{(n-2)^2}{4}{w}_{i\infty} + \Delta_{\mathbb{S}^{n-1}}{w_{i\infty}} + v_{\infty}^{\frac{4}{n-2}}\left(n\langle {\mathcal{W_\infty}},\Lambda\rangle \Lambda_{i} + \frac{n(n-2)}{4}{w}_{i\infty}\right) = 0
\end{equation}
in the sense of distributions. Here $\mathcal W_\infty=(w_{1\infty},\ldots,w_{d\infty})$. If $\tau_i$ is a real number, then the boundary condition becomes $ {\mathcal W}_\infty(\tau_i,\theta)=0$, otherwise we use the decay prescription 
\begin{equation}\label{eq013}
| {\mathcal W}_\infty(t,\theta)|\leq e^{-\delta t}.
\end{equation}
It is important to point out that the condition
\eqref{eq003} implies that
\begin{equation}\label{eq014}
\sup_{\theta \in \mathbb{S}^{n-1}}| {\mathcal W}_\infty(0,\theta)|=1.
\end{equation}
Thus ${\mathcal W}_\infty$ is non trivial. Decompose each ${{w}}_{i\infty}$ as 
\begin{equation*}
{w_{i\infty}}(t,\theta) = \sum_{j=n+1}^\infty{ {w}}_{ij\infty}(t)\phi_{j}(\theta).
\end{equation*}
Thus, each $w_{ij\infty}$ satisfies the equation
\begin{equation}\label{eq008}
\frac{d^{2}{w}_{ij\infty}}{dt^{2}} - \frac{(n-2)^2}{4}{w}_{ij\infty} -\lambda_j{w_{ij\infty}} + v_{\infty}^{\frac{4}{n-2}}\left(n\langle {\mathcal{W}_{j\infty}},\Lambda\rangle \Lambda_{i} + \frac{n(n-2)}{4}{w}_{ij\infty}\right) = 0,
\end{equation}
where $\mathcal W_{j\infty}=(w_{1j\infty},\ldots,w_{dj\infty})$. Let us proof that this is a contradiction. 

\noindent{\bf Claim 1.} If $\tau_1=-\infty$ ($\tau_2=+\infty$), then $\mathcal W_\infty$ decays exponentially as $t\rightarrow-\infty$ (respectively, $t\rightarrow+\infty$).

We have to consider some cases. First consider the case $0<v_\infty<1$. 

Since $\delta < \frac{n+2}{2}$ we can choose $\eta>0$ such that $\delta^2 < \left(\frac{n+2}{2}\right)^2 - 4\eta$ to get
$$| {\mathcal W}_{j\infty}(t,\theta)|\leq e^{-\delta t}< e^{-\left(\left(\frac{n+2}{2}\right)^2-4\eta\right)^{1/2}t},\;\;\mbox{ for }t<0.$$ 
Thus, by Lemma \ref{lema2}, we find that for $\varepsilon>0$ small enough, ${\mathcal{W}_{\infty}}$ decays exponentially as t goes to $-\infty$. Similarly, if $\tau_{2} = + \infty$,  we get that ${\mathcal{W}_{\infty}}$ decays exponentially as $t \rightarrow \infty$. 

Now, if $v_\infty\equiv 0$, then by \eqref{eq008} we get  that
\begin{equation}\label{eq007}
\frac{d^{2}{w}_{ij\infty}}{dt^{2}} - \frac{(n-2)^2}{4}{w}_{ij\infty} - \lambda_j{w_{ij\infty}}= 0.
\end{equation}
Thus, $w_{ij\infty}=c_je^{\pm\gamma_j t}$, with $\gamma_j^2=\left(\frac{n-2}{2}\right)^{2}+\lambda_j\geq \left(\frac{n+2}{2}\right)^{2}$. By \eqref{eq013} with $\delta<\frac{n+2}{2}$ we conclude that $w_{ij\infty}=c_je^{\gamma_j t}$ for $t<0$. Using that $\delta\geq-\frac{n+2}{2}$ we find that $w_{ij\infty}=c_je^{-\gamma_jt}$ as $t\rightarrow+\infty$. Therefore, $\mathcal W_{j\infty}$ decays exponentially as $t\rightarrow-\infty$ and as $t\rightarrow+\infty$.

Finally, suppose that $v_\infty(t)=(\cosh(t-t_0))^{\frac{2-n}{2}}$. 
Since $\displaystyle\lim_{t\rightarrow \pm\infty}v_\infty^{\frac{4}{n-2}}(t)=0$, we deduce that 
$$w_{ij\infty}\sim e^{\pm\gamma_jt}\;\;\mbox{ as }t\rightarrow\pm\infty.$$
Again, the decay imposed to $\mathcal W_\infty$ implies that $w_{ij\infty}$ decays exponentially as $t\rightarrow\pm\infty$.

Now, we can multiply \eqref{eq008} by $w_{ij\infty}$ and apply integration by parts. In fact, if both $\tau_i$, $i=1,2$, are finite, we just have to use the fact that $\mathcal W_\infty(\tau_i,\theta)=0$. Otherwise, we use the Claim 1 to assure that there are no boundary terms.

Using that $\lambda\geq 2n$ and $0<v_\infty<1$ we obtain  
$$
\begin{array}{ccl}
0& = & \displaystyle\int_{\tau_{1}}^{\tau_{2}} \left( |\dot{\mathcal W}_{j\infty}|^2 + \frac{(n-2)^{2}}{4}|{\mathcal W_{j\infty}}|^{2} + \lambda_{j}|{\mathcal W_{j\infty}}|^{2}- v_{\infty}^{\frac{4}{n-2}}\left(n\langle {\mathcal{W}}_{\infty},\Lambda\rangle^{2} \right.\right.\\
& & \left.\left.\displaystyle + \frac{n(n-2)}{4}|{\mathcal W_{j\infty}}|^{2}\right) \right)dt\\
& \leq & \displaystyle\int_{\tau_1}^{\tau_2}\left(|\dot {\mathcal W}_{j\infty}|^2+\frac{n+2}{2}|\mathcal W_{j\infty}|^2\right)dt,
\end{array}
$$
Therefore ${\mathcal{W}}_\infty \equiv 0$, which is a contradiction with \eqref{eq014}
\end{proof}

Estimates for the full Hölder norm follows
by standard scaling arguments. Thus, given $\overline f=(\overline f_1, \ldots,\overline f_d)\in C_{\delta;d}^{0,\alpha}(D_R)$, the solution $\mathcal W_T$ of  \eqref{eq011} which verifies \eqref{eq012}, belongs to $C_{\delta,d}^{2,\alpha}(D_R^T)$. Moreover, for $\delta\in\left(-\frac{n+2}{2},\frac{n+2}{2}\right)$ there exists a constant $\varepsilon_0>0$ such that for all $\varepsilon\in(0,\varepsilon_0]$ we have
	\begin{equation*}
\|\overline{\mathcal W}_T\|_{C_{\delta;d}^{2,\alpha}(D_R^T)}\leq C\|\overline f\|_{C_{\delta;d}^{0,\alpha}(D_R^T)},
\end{equation*}
where $C>0$ is a constant which does not depends on $\varepsilon$, $R$ and $T$. Thus, letting $T\rightarrow+\infty$ we obtain the existence of a solution $\overline{\mathcal W}$ to \eqref{eq010} which verifies the
estimate
	\begin{equation*}
\|\overline{\mathcal W}\|_{C_{\delta;d}^{2,\alpha}(D_R)}\leq C\|\overline f\|_{C_{\delta;d}^{0,\alpha}(D_R)},
\end{equation*}

To summarize the analysis for the high frequencies we introduce the notations
$$\begin{array}{rcl}
C_{\mu;d}^{k,\alpha}(B_r(0)\backslash\{0\})^\perp & := & \displaystyle\left\{u\in C_{\mu;d}^{k,\alpha}(B_r(0)\backslash\{0\});\int_{\mathbb S^{n-1}}u_i(s\theta)\phi_j(\theta)d\theta=0,\right.\\
& & \displaystyle \left.\mbox{ for all }j\in\{0,\ldots,n\}, s\in(0, r] \mbox{ and }i\in\{1,\ldots,d\} \right\}.
\end{array}$$
and
$$\left[C_{\mu;d}^{k,\alpha}(B_r(0)\backslash\{0\})^\perp\right]_0:=\left\{u\in C_{\mu;d}^{k,\alpha}(B_r(0)\backslash\{0\})^\perp;u(x)=0\mbox{ for all } x \mbox{ with }|x|=r\right\}.$$

The previous analysis tell us that, for all $\mu\in(-n,2)$, the operator 
$$\mathcal L_{\varepsilon,R}:\left[C_{\mu;d}^{2,\alpha}(B_1(0)\backslash\{0\})^\perp\right]_0\rightarrow C_{\mu-2;d}^{0,\alpha}(B_1(0)\backslash\{0\})^\perp,$$
is an isomorphism with the inverse bounded independently of $\varepsilon$ and $R$.

\subsubsection{Low frequencies: $j=0,\ldots,n$}

We start by considering the projection of our original problem \eqref{sol2} along the eigenfunction $\phi_0$, obtaining
\begin{equation*}
\left\{
\begin{array}{rcl}
\mathcal L_{\varepsilon,0}^{i}(\mathcal{W}_0) & = &  f_{i0} \;\;\mbox{ in } \;\;(\log R,+\infty)\\
w_{i0}(\log R) & = & 0,
\end{array}
\right.
\end{equation*}
where 
\begin{equation}\label{eq016}
\mathcal L_{\varepsilon,0}^{i}(\mathcal W_0)=\frac{d^{2}{w}_{i0}}{dt^{2}} - \frac{(n-2)^2}{4}{w}_{i0} + v_{\varepsilon}^{\frac{4}{n-2}}\left(n\langle {\mathcal{W}_0},\Lambda\rangle \Lambda_{i} + \frac{n(n-2)}{4}{w}_{i0}\right).
\end{equation}

In this case, the potential has the wrong sign. Then, we need to use a different approach in order to provide existence. Following \cite{MP}, we suppose that $f_{i0}$ is at least continuous and we extend it to the whole $\R$. For any $T>\log R$, we consider the auxiliary backward Cauchy problem
\begin{equation}\label{eq017}
\left\{
\begin{array}{rcl}
\mathcal L_{\varepsilon,0}^{i}(\mathcal{W}_{0}) & = &  f_{i0} \;\;\mbox{ in } \;\;(-\infty,T)\\
w_{i0}(\log R) & = & 0.\\
w_{i0}(T) & = & 0,\\
\end{array}
\right.
\end{equation}

It is easy to see that there exists a unique solution $\mathcal W_T^0$ to \eqref{eq017}. Next, we define the spaces $C^k_{\delta;d}(-\infty,T)$ similarly as in \eqref{eq023}.

\begin{proposition}\label{propo001}
 Let $\mathcal W_{T0}\in C^2_{\delta;d}(-\infty,T)$ be a solution of \eqref{eq017} with $f_{i0}\in C_{\delta}^0(-\infty,T)$. Then, for every $\delta>\frac{n-2}{2}$, there exist constants $\varepsilon_0>0$ and $C>0$, such that for every $(0,\varepsilon_0]$, $R>0$, $T>\log R$ we have
 \begin{equation}\label{eq022}
  \|\mathcal W_{T0}\|_{C_{\delta;d}^0(\log R,T)}\leq C \|f_0\|_{C_{\delta;d}^0(\log R,T)}.
 \end{equation}
\end{proposition}
\begin{proof}
 The proof is by contradiction. If \eqref{eq022} does not hold, then it is possible to find a sequence $(\varepsilon_k,R_k,T_k,\mathcal W_{T_k0},f_{0k})$ such that
 \begin{enumerate}
  \item $\varepsilon_k\rightarrow 0$ as $k\rightarrow +\infty$.
\item $\|f_0\|_{C_{\delta;d}^0(\log R_k,T_k)}=1$.
\item $\|\mathcal W_{T_k0}\|_{C_{\delta;d}^0(\log R_k,T_k)}\rightarrow+\infty$ as $k\rightarrow+\infty$.
  \item $$\left\{
\begin{array}{rcl}
\mathcal L_{\varepsilon,0}^{i}(\mathcal{W}_{T_k0}) & = &  f_{ik0} \;\;\mbox{ in } \;\;(-\infty,T_k)\\
w_{i0}(T_k) & = & 0,\\
\dot w_{i0}(T_k) & = & 0.\\
\end{array}
\right.
$$
 \end{enumerate}
 
 Choose $t_k\in[\log R_k,T_k]$ such that
 $$\|\mathcal W_{T_k0}\|_{C_{\delta;d}^0(\log R,T)}=e^{\delta t_k}|\mathcal W_{T_k}^0(t_k)|=:A_k.$$
Define
 $\mathcal W_k(t):=A_k^{-1}e^{\delta t_k}\mathcal W_{T_k0}(t+t_k)$ and $f_k(t):=A_k^{-1}e^{\delta t_k}f_0(t+t_k),$ for $t\in[\log R_k-t_k,T_k-t_k]$. From these definitions it follows that
$$\left\{
\begin{array}{rcl}
\mathcal L_{\varepsilon,0}^{i}(\mathcal{W}_{k}) & = &  f_{ik} \;\;\mbox{ in } \;\;(\log R_k-t_k,T_k-t_k)\\
\mathcal W_{k}(T_k-t_k) & = & 0,\\
\dot {\mathcal W}_{k}(T_k-t_k) & = & 0,\\
\end{array}
\right.
$$
\begin{equation}\label{eq019}
 \sup_{t\in[\log R_k-t_k,T_k-t_k]}e^{\delta t}|\mathcal W_k(t)|=|\mathcal W_k(0)|=1
\end{equation}
and 
$$\sup_{t\in[\log R_k-t_k,T_k-t_k]}e^{\delta t}|f_k(t)|\rightarrow0\;\;\mbox{ as }\;\;k\rightarrow+\infty.$$

As in the proof of Proposition \ref{prop001}, up to a subsequence, we have $\log R_k-t_k\rightarrow \tau_1\in[-\infty,0)$ and $T_k-t_k\rightarrow\tau_2\in(0,+\infty]$. Like before we get a function $\mathcal W_\infty$ and a Fowler solution $v_\infty$ such that $\mathcal W_k\rightarrow W_\infty$, $v_{\varepsilon_k}(t+t_k)\rightarrow v_\infty(t)$ and
\begin{equation}
\frac{d^{2}{w}_{i\infty}}{dt^{2}} - \frac{(n-2)^2}{4}{w}_{i\infty} + v_{\infty}^{\frac{4}{n-2}}\left(n\langle {\mathcal{W}_\infty},\Lambda\rangle \Lambda_{i} + \frac{n(n-2)}{4}{w}_{i\infty}\right)=0
\end{equation}
in $t\in(\tau_1,\tau_2)$. By \eqref{eq019} we get that $\mathcal W_\infty(0)=1$ and then $\mathcal W_\infty$ is non trivial.

If $\tau_2<+\infty$, then the Cauchy data for the limit problem are given by $\mathcal W_\infty(\tau_2)=\dot{\mathcal W}_\infty(\tau_2)=0$. This implies at once that $\mathcal W_\infty\equiv 0$.

\noindent{\bf Case 1.} $v_\infty\equiv 0$.

In this case, we have
\begin{equation}
\frac{d^{2}{w}_{i\infty}}{dt^{2}} - \frac{(n-2)^2}{4}{w}_{i\infty}=0,
\end{equation}
which implies that $w_{i\infty}(t)=c_ie^{\pm\frac{n-2}{2}t}$. But, the condition $|\mathcal W_\infty|\leq e^{-\delta t}$ and the fact that $\delta>\frac{n-2}{2}$ implies that $w_{i\infty}\equiv 0$.

\noindent{\bf Case 2.} $v_\infty>0$. 

In this case we have $\displaystyle\lim_{t\rightarrow \pm\infty}v_\infty(t)=0$. Thus, we get
$$\mathcal W_\infty(t)\sim e^{\pm\frac{n-2}{2}t}.$$
Again, in the same way as in the Case 1 we get that $\mathcal W_{\infty}\equiv 0$.

In both case we have a contradiction.\qedhere
\end{proof}

Since the estimate \eqref{eq022} is independent of the parameter $T>\log R$, we let $T\rightarrow+\infty$ and we obtain a
function $\mathcal W_0$ which verifies the identity
$$\mathcal L^i_{\varepsilon,0}(\mathcal W_0)=f_0\;\;\mbox{ in }\mathbb{R}$$
together with the estimate
$$\|\mathcal W_0\|_{C^0_{\gamma;d}(\log R,+\infty)}\leq C\|f_0\|_{C^0_{\gamma;d}(\log R,+\infty)},$$
where $\gamma>\frac{n-2}{2}$ and the positive constant $C$ does not depend on $\varepsilon$ and $R$. Moreover, we notice that if $f_0\in C^{0,\alpha}_{\gamma;d}(\log R,+\infty)$ then $\mathcal W_0\in C^{2,\alpha}_{\gamma;d}(\log R,+\infty)$ and there exists a positive constant $C$, which does not depend on $\varepsilon$ and $R$, such that 
$$\|\mathcal W_0\|_{C^{2,\alpha}_{\gamma;d}(\log R,+\infty)}\leq C\|f_0\|_{C^{0,\alpha}_{\gamma;d}(\log R,+\infty)},$$
for every $\varepsilon\in(0,\varepsilon)$.

Now we are ready to treat the projection of \eqref{sol2} along the eigenfunction $\phi_j$, with $j=1,\ldots,n$
\begin{equation}\label{eq020}
\left\{
\begin{array}{rcl}
\mathcal L_{\varepsilon,j}^{i}(\mathcal{W}_{j}) & = &  f_{i0} \;\;\mbox{ in } \;\;(\log R,+\infty)\\
w_{i}^j(\log R) & = & 0.
\end{array}
\right.
\end{equation}

Proceeding in the same manner as in the case $j=0$, we are led to consider the auxiliary backward Cauchy problem
\begin{equation}\label{eq021}
\left\{
\begin{array}{rcl}
\mathcal L_{\varepsilon,j}^{i}(\mathcal{W}) & = &  f_{ij} \;\;\mbox{ in } \;\;(-\infty,T)\\
\mathcal W(T) & = & 0\\
\dot{\mathcal W}(T) & = & 0,
\end{array}
\right.
\end{equation}
where, with the usual abuse of notations, $f_{ij}$ denotes the extension to the whole $\mathbb R$ of the former data. The Cauchy-Lipschitz Theorem guarantees the existence of a solution $\mathcal W_{Tj}$ to \eqref{eq021} for every $j=1,\ldots,n$.

\begin{proposition}\label{prop002}
	Let $\mathcal W_{Tj}\in C_{\delta;d}^2(\log R,T)$ be a solution of \eqref{eq021} with $f_{ij}\in C_\delta^0(-\infty,T)$, for $j=1,\ldots,n$. Then, for every $\delta>\frac{n}{2}$, there exist constants $\varepsilon_0>0$ and $C>0$, such that for every $(0,\varepsilon_0]$, $R>0$, $T>\log R$ we have
	\begin{equation}\label{eq018}
	\|\mathcal W_{Tj}\|_{C_{\delta;d}^0(\log R,T)}\leq \|f_j\|_{C_{\delta;d}^0(\log R,T)}.
	\end{equation}
\end{proposition}
\begin{proof}
	The proof of this proposition is identical to the one of Proposition \ref{propo001}. The slightly diferent choice of the range of the weights is due to the fact that for $j=1,\ldots,n$	the eigenvalue $\lambda_j$ is equal $n-1$ whereas for $j=0$ we had $\lambda_0=0$.
\end{proof}

Arguing as in the case $j=0$, we let now $T\rightarrow+\infty$ and we get, for every $j=1,\ldots,n$, a solution $\mathcal W_j$ to the equation 
$\mathcal L^i_{\varepsilon,R}(\mathcal W_j)=f_{ij}\;\;\mbox{ in }\mathbb R$. Moreover, if $f_j\in C^{0,\alpha}_{\gamma;d}(\log R,+\infty)$ with $\gamma>\frac{n}{2}$, then $\mathcal W_j\in C^{2,\alpha}_{\gamma;d}(\log R,+\infty)$ and
$$\|\mathcal W_j\|_{C^{2,\alpha}_{\gamma;d}(\log R,+\infty)}\leq C\|\mathcal W_j\|_{C^{0,\alpha}_{\gamma;d}(\log R,+\infty)},$$
for some constant $C>0$ independently of $\varepsilon$ and $R$.

Note that there is no reason why the solutions $\mathcal W_j$, for all $j=0,\ldots,n$, satisfy the boundary condition at $t=\log R$. But, as in \cite{MP}, we can use the {\it Jacobi fields}, which are solutions of the equation $\mathcal L_{\varepsilon}(\Phi)=0$, to get solutions $\overline{\mathcal W}_j(\log R)=0$ and still satisfing the estimates in Propositions  \ref{propo001} and \ref{prop002} (see \cite{MP} for more details).
\subsubsection{A right inverse}

We can now collect all the results of the previous sections to state the following proposition. Before, we define the projection $\pi_r''$ onto the high frequencies space by the formula
$$\pi_r''(\varphi)(r\theta)=\sum_{j=n+1}^\infty\varphi_j(r)\phi_j(\theta),$$
where
$$\varphi(r\theta)=\sum_{j=0}^\infty\varphi_j(r)\phi_j(\theta).$$

We often will write the functions spaces with a superscript symbol $\perp$ to indicate that $\pi_r''(\phi)=\phi$ for all function $\phi$ belongs to that space, for example
$$C^{k,\alpha}(\mathbb S_r^{n-1})^\perp:=\{\phi\in C^{k,\alpha}(\mathbb S_r^{n-1});\pi_r(\phi)=\phi\}.$$

\begin{proposition}\label{propo003}
Fix $\mu\in\left(1,2\right)$. There exists an $\varepsilon_0>0$ such that, for all $\varepsilon\in(0,\varepsilon_0)$, there exists an operator
$$G_{\varepsilon,R,r}:C^{0,\alpha}_{\mu-2;d}(B_r(0)\backslash\{0\})\rightarrow C^{2,\alpha}_{\mu;d}(B_r(0)\backslash\{0\}),$$
such that for each $f\in C^{0,\alpha}_{\mu-2;d}(B_r(0)\backslash\{0\})$, the function $\mathcal W=G_{\varepsilon,R,r}(f)$ solves the equation
$$\left\{\begin{array}{lcl}
\mathcal L_{\varepsilon,R}(\mathcal W)=f & \mbox{ in } & B_r(0)\backslash\{0\}\\
\pi_r''(\mathcal W|_{\partial B_1(0)})=0 & \mbox{ on } & \partial B_r(0),
\end{array}\right.$$
and satisfies the bounded
$$\|f\|_{(0,\alpha),\mu-2;d}\leq C\|\mathcal W \|_{(2,\alpha),\mu;d},$$
for some constant $C>0$ which does not depends on $\varepsilon$ and $R$. Moreover, if $f\in C^{0,\alpha}_{\mu-2;d}(B_r(0)\backslash\{0\})^\perp $, then $\mathcal W\in C^{2,\alpha}_{\mu;d}(B_r(0)\backslash\{0\})^\perp$ and we may take $\mu \in (-n,2)$.
\end{proposition}

In fact, we will work with the function $u_{\varepsilon,R,a}$, and so, we need to find an inverse to $\mathcal L_{\varepsilon,R,a}$ with norm bounded independently of $\varepsilon$, $R$, $a$ and $r$. But this is the content of the next corollary, whose proof is a perturbation argument.

\begin{corollary}\label{inver} Let $ \mu \in (1,2)$, $\alpha \in (0,1)$ and $\varepsilon_0>0$ given by Proposition \ref{propo003}. Then for all $\varepsilon\in(0,\varepsilon_0)$, $R>0$, $a\in\mathbb R^n$ and $0<r<1$ with $|a|r\leq r_0$ for some $r_0\in(0,1)$ fixed, there is an operator
	\begin{equation*}
	G_{\varepsilon,R,r,a}: C^{0,\alpha}_{\mu -2;d}(B_{r}(0)\backslash \{0\})\rightarrow C^{2,\alpha}_{\mu;d}(B_{r}(0)\backslash\{0\})
	\end{equation*}
	with norm bounded independently of $\varepsilon$,  $R$, $r$ and $a$, such that for each $f \in C^{0,\alpha}_{\mu-2;d}(B_{r}(0)\backslash \{0\})$, the function $\mathcal W := G_{\varepsilon,R,r,a}(f)$ solves the system
	\begin{align*}\left\{
	\begin{array}{l}
	\mathcal{L}_{\varepsilon,R,a}(\mathcal W) = f \quad \mbox{ in } \quad B_{r}(0)\backslash \{0\} \\ 
	\pi_{r}''(\mathcal W|_{\mathbb{S}_r^{n-1}}) = 0 \quad \mbox{ on } \quad  \partial B_{r}(0)\backslash \{0\}.
	\end{array}\right.
	\end{align*}
	
\end{corollary}
\begin{proof}
First, we notice that
$$(\mathcal L_{\varepsilon,R,a}^i-\mathcal L_{\varepsilon,R}^i)(\mathcal W)=\left(n\langle \Lambda,\mathcal W\rangle\Lambda_i+\frac{n(n-2)}{4}w_i\right)(u_{\varepsilon,R,a}^{\frac{4}{n-2}}-u_{\varepsilon,R}^{\frac{4}{n-2}}).$$
Using this and the properties of the Fowler solution (see \cite{MP}) we obtain that
$$\|(\mathcal L_{\varepsilon,R,a}-\mathcal L_{\varepsilon,R})\mathcal W\|_{(0,\alpha),\mu-2,r;d}\leq cr|a|\|\mathcal W\|_{(2,\alpha),\mu,r;d},$$
for some positive constant $c>0$ which does not depend on $\varepsilon$, $R$, $a$ and $r$. The result follows by a perturbation argument.
\end{proof}

\subsection{Analysis in \texorpdfstring{$M_r:=M\backslash B_r(p)$}{Lg}}\label{sec006}

Since we are assuming that the metric $g$ is nondegerenate in the sense that $L_{g_{0}}: C^{2,\alpha}(M)^d\rightarrow C^{0,\alpha}(M)^d$ is surjective for some $\alpha \in (0,1)$, see Definition \ref{def001}, we can prove the following proposition. But first, we remember from Lemma 13.23 in \cite{jleli} that for $\nu\in(1-n,2-n)$ and $0<2r<s$ there exists an operator 
$P:C_{\nu-2}^{0,\alpha}(A_{r,s})\rightarrow C_{\nu}^{2,\alpha}(A_{r,s})$, where $A_{r,s}=B_s(0)\backslash B_r(0)\subset\mathbb R^n$, such that
\begin{equation}\label{eq029}
    \left\{\begin{array}{rcl}
     \Delta P(f)=f & \mbox{ in } & A_{r,s}  \\
     w = 0 & \mbox{ on } & \partial B_s(0)\\
     w\in\mathbb R & \mbox{ on } & \partial B_r(0)
\end{array}\right.,
\end{equation}
and
$$\|P(f)\|_{C_{\nu}^{2,\alpha}(A_{r,s})}\leq C\|f\|_{C_{\nu-2}^{0,\alpha}(A_{r,s})}.$$

Using this operator we can prove the following.

\begin{proposition}\label{propo004}
Fix $\nu \in (1-n,2-n)$. There exists $r_{2}>0$ such that, for all $r \in (0,r_{2})$ we can define an operator 
	$G_{r, g_{0}} : C^{0,\alpha}_{\nu-2;d}(M_{r})\rightarrow C^{2,\alpha}_{\nu;d}(M_{r}),$
such that, for all $f \in C^{0,\alpha}_{\nu-2;d}(M_{r})$ the function $w = G_{r, g_{0}}(f)=(w_1,\ldots,w_d)$ solves $
	L_{g_{0}}(w) = f$ in $M_{r}$
with $w_{i}\in \R$ constant along $\partial B_{r}(p)$. In addition
\begin{equation*}
	\|G_{r,g_{0}}(f)\|_{C^{2,\alpha}_{\nu;d}(M_{r})} \leq C\|f\|_{C^{0,\alpha}_{\nu;d}(M_{r})}
\end{equation*}
where $C>0$ does not depend on $r$.
\end{proposition}
\begin{proof}
Using a perturbation argument we can find an operator, still denoted by $P:C_{\nu-2}^{0,\alpha}(A_{r,r_1})\rightarrow C_{\nu}^{2,\alpha}(A_{r,r_1})$ such that it solves a similar equation as in \eqref{eq029} with the laplacian $\Delta$ replaced by $\mathcal L_{g_0}^i$, see \eqref{li}.

Let $f=(f_1,\ldots,f_d)\in C^{0,\alpha}_{\nu-2;d}(M_r)$ e define a function $\eta\in C^{2,\alpha}_{\nu;d}(M_r)$ by $\eta:=\chi_1(P(f_1),\ldots,P(f_d))$, where $\chi_1$ is a smooth, radial function equal to 1 in $B_{r_1/2}(p)$ and vanishing in $M_{r_1}$. Note that
\begin{equation}\label{eq030}
    \|\eta\|_{C^{2,\alpha}_{\nu;d}(M_r)}\leq C\|f\|_{C^{0,\alpha}_{\nu-2;d}(M_r)},
\end{equation}
for some positive constant $C$ which does not depend on $r$ and $r_1$. Now, define $h:=f-\mathcal L_{g_0}(\eta)$ and note that it is supported in $M_{r_1/2}$. Thus we can consider $h\in C^{0,\alpha}(M)$ with $h\equiv 0$ in $B_{r_1/2}(p)$. Therefore, using \eqref{eq030} we obtain
$$\|h\|_{C_{;d}^{0,\alpha}(M)}\leq C_{r_1}\|f\|_{C_{\nu-2;d}^{0,\alpha}(M_r)},$$
with $C_{r_1}>0$ independent of $r$.

Since $g_0$ is nondegenerate, then $\mathcal L_{g_0}:C_{;d}^{2,\alpha}(M)\rightarrow C_{;d}^{0,\alpha}(M)$ has a bounded inverse. Thus, define $\eta_1:=\chi_2(\mathcal L_{g_0})^{-1}(h)$, where $\chi_2$ is a smooth radial function equal to 1 in $M_{2r_2}$ and vanishing in $B_{r_2}(p)$ for some $r_2\in(r,r_1/4)$. By the previous estimates we get that
$$\|\eta_1\|_{C^{2,\alpha}_{\nu;d}(M_r)}\leq C\|f\|_{C^{0,\alpha}_{\nu-2;d}(M_r)}.$$

Now, define a map $F_{r}:C^{0,\alpha}_{\nu-2;d}(M_r)\rightarrow C^{2,\alpha}_{\nu-2;d}(M_r)$ as $F_r(f)=\eta+\eta_1$. Now the result follows by a perturbation argument.

\end{proof}


\section{Nonlinear Analysis}\label{sec009}

Now that we already have the right inverses, we will solve the system \eqref{S} in the punctured geodesic ball and in its complement. We will do that using a perturbation technique and a fixed point argument by using the expansion \eqref{eq057} and the right inverses constructed in Corollary \ref{inver} and in Proposition \ref{propo004}. In order to perform the gluing procedure in Section \ref{sec005} the estimates obtained in the interior analysis have to be compatible with the one obtained in the exterior analysis. Also, we are dealing with a system and the approximate solution in the interior analysis is the Fowler-type solution $\mathcal U_{\varepsilon,R,a}=u_{\varepsilon,R,a}\Lambda$, which does not give us enough parameters to control the low frequency space in each equation of the system in the gluing procedure. It turns out that the parameter $R\in\mathbb R$ controls the constant functions space and $a\in\mathbb R
^n$ controls the coordinates functions space, but since the equation \eqref{S} has $d$ equations we will need distinct parameters to controls these spaces in each equation. To overcome this difficult we need to  use an auxiliary function $h$, see \eqref{eq031}.

\subsection{Interior Analysis}\label{sec002}

In this section we will use the hypothesis \eqref{H1} to find a family of solution to \eqref{S} on a punctured ball, with Dirichlet boundary data. First, we recall from \cite{jleli} and \cite{almir} that for any $\mu\leq 2$ and $r,\alpha\in(0,1)$, give a function $\phi\in C^{2,\alpha}(\mathbb S_r^{n-1})^\perp$ there exists a function $v_\phi\in C_\mu^{2,\alpha}(B_r(0)\backslash\{0\})^\perp$ such that
\begin{equation}\label{eq054}
    \left\{\begin{array}{lcl}
\Delta v_\phi=0 & \mbox{ in } & B_r(0)\backslash\{0\}\\
\pi_r''(v_\phi|_{\mathbb S_r^{n-1}})=\phi & \mbox{ on} & \partial B_r(0)
\end{array}\right.
\end{equation}
and
\begin{equation}\label{po}
\|v_\phi\|_{(2,\alpha),\mu,r}\leq Cr^{-\mu}\|\phi\|_{(2,\alpha),r}
\end{equation}
for some positive constant $C>0$ which does not depend on $r$.  This give us a well known operator called Poisson operator which is well understood and has useful properties that allows us to match the boundary Cauchy data, in order to successfully gluing the solutions that we will construct. See \cite{jleli}, \cite{MR1763040} and \cite{almir} for details.

Also,  we define a function $h$ as
\begin{equation}\label{eq031}
\begin{array}{rcl}
     h & = & \displaystyle\chi(x)\left(1+r^{20-n}|x|^4-\frac{4}{5}r^{19-n}|x|^5-r^{24-n}+\frac{4}{5}r^{24-n}\right)\times\\
     & &\times \left((\eta,0) +|x|(\langle  A_1,\theta\rangle,\ldots,\langle A_{d-1},\theta\rangle,0)\right)
\end{array}
\end{equation}
where $\chi$ is a smooth radial function equal to 1 in $B_{r}\backslash B_{r^{5}}$ and equal to zero in $B_{r^{10}}$, $\eta\in\mathbb R^{d-1}$ and $A_i\in\mathbb R^n$, for $i=1,\ldots,d-1$, are constants to be chosen later. Here $\theta=x/|x|\in\mathbb S^{n-1}$. Note that $h=O(|x|^4)$ and if $|x|=r$, we have $h(x)=\left(\eta,0)+r(\langle A_1,\theta\rangle,\ldots,\langle A_{d-1},\theta\rangle,0\right)$ and $\partial_rh(x)=(\langle A_1,\theta\rangle,\ldots,\langle A_{d-1},\theta\rangle,0)$.

Now, for each $\phi_{i} \in C^{2,\alpha}(\mathbb{S}^{n-1}_{r})^\perp$, $i=1,\ldots,d$, consider $v_{\phi_{i}}\in C^{2,\alpha}(B_{r}(0)\backslash\{0\})^\perp$ and $\mathcal{V}_{\phi} = (v_{\phi_{1}},\ldots,v_{\phi_{d}})$. The main goal in this section is to solve the system of PDEs 
\begin{equation}\label{SS}
H_{g_0}(\beta^{-1}(\mathcal{U}_{\varepsilon,R,a} +h+\mathcal V_\phi+ \mathcal{V})) = 0,
\end{equation}
in $B_{r}(0)\backslash \{0\}$, for some suitable parameters $0 < r \leq 1$, $\varepsilon > 0$, $R > 0$ and $a \in \R^{n}$, with $\mathcal{U}_{\varepsilon,R,a} +h+\mathcal V_{\phi}+ \mathcal{V} > 0$ and prescribed Dirichlet data, where $H_{g_0}$ is defined in Section \ref{sec001}. Here, $\beta$ is defined in the next paragraph.

To solve this problem we will use the fixed point method on Banach space. We will use the right inverse given by Corollary \ref{inver}.
However, since the right inverse is defined in the weighted H\"{o}lder spaces for some suitable weights, work with a general metric can hamper our calculations. To bypass this problem, instead of consider the metric $g_0$, we know that it is possible to find a positive smooth function $\beta \in C^{\infty}(M)$ such that the conformal metric $g = \beta^{-\frac{4}{n-2}}g_0$ has a good behaviour in the neighborhood of $p$. In fact, near the point $p$, in $g$-normal coordinates we have $\det g_{ij}=1+O(|x|^N)$, for $N$ big enough, $R_g=O(|x|^2)$ and $\beta=1+O(|x|^2)$ (see \cite[Theorem 5.1]{Lee_Parker}).  We will work in this coordinates, in the ball $B_{r_1}(p)$ with $0 < r_{1} \leq 1$ fixed. 

\subsubsection{Fixed point argument}

Due to the decay of the metric and of the scalar curvature, in this section we will restrict ourselves to work only in dimensions $3 \leq n \leq 5$. 


One easily checks that to find a solution to the system \eqref{SS}, using the conformal change \eqref{eq027}, it is equivalent to solve the following nonlinear PDE system
\begin{equation}\label{2.4}
\begin{array}{rcl}
\mathcal{L}_{\varepsilon, R, a}^{i}(\mathcal{V}) & = & \displaystyle (\Delta - \Delta_g)(u_{\varepsilon, R, a}\Lambda_i +h_i+ v_{\phi_i}+ v_i )+\Delta h_i\\
\\
& & + c_nR_g(u_{\varepsilon,R,a}\Lambda_i +h_i+v_{\phi_i}+ v_i ) -n(n-1)c_n u_{\varepsilon,R,a}^{\frac{4}{n-2}}(h_i+v_{\phi_i})\\
\\
& & \displaystyle +\beta^{-\frac{4}{n-2}}\sum_{j=1}^{2}\left(A_{ij}-c_nR_{g_0}\delta_{ij}\right)(u_{\varepsilon,R,a}\Lambda_{j}+h_i + v_{\phi_j}+ v_{j})\\
\\
& & - Q_{\varepsilon,R,a}^{i}(h+\mathcal{V}_\phi + \mathcal{V}) - nu_{\varepsilon,R,a}^{\frac{4}{n-2}}\langle h+\mathcal{V}_{\phi}, \Lambda \rangle\Lambda_i
\end{array}
\end{equation}
for $i=1,\ldots,d$, where $h=(h_1,\ldots,h_d)$, $\mathcal L^i_{\varepsilon,R,a}$ is defined in \eqref{eq002} and $Q_{\varepsilon,R,a}^{i}$ in \eqref{resto}. Here $c_n=\frac{n-2}{4(n-1)}$.

As previously mentioned, since our strategy is to reduce the problem of finding a solution of the system to a fixed point problem, we need to show that the right hand side of \eqref{2.4} belongs to $C^{0,\alpha}_{\mu-2;d}(B_r(0)\backslash\{0\})$.

\begin{remark}\label{rem001}
Throughout the rest of this work we will consider $d_3 =0$, $d_4=d_5=1$ and $r_\varepsilon=\varepsilon^s$ with $(d_n + 24/25)^{-1} < s < 4(d_n - 2 + 3n/2)^{-1}$. Also, we will consider $$R^{\frac{2-n}{2}}=2(1+b)\varepsilon^{-1}.$$
\end{remark}	

	These restrictions allow us to obtain good estimates for the approximate solution in a small ball near the singularity. They are importants to ensure the estimates in the next lemma.  First, we notice that we get existence of positive constants $C_{1}$ and $C_{2}$ such that 
\begin{equation}\label{2.6}
C_{1}\varepsilon |x|^{\frac{2-n}{2}} \leq u_{\varepsilon, R, a}(x) \leq C_{2}|x|^{\frac{2-n}{2}}
\end{equation}
for every $x \in B_{r_{\varepsilon}}(0)\backslash \{0\}$ and all $a\in\mathbb R^n$ with $|a|r_\varepsilon\leq 1/2$. Using this we obtain the next lemma. 
\begin{lemma}\label{re} Let $\mu \in (1, 2)$. There exists $\varepsilon_0 \in (0,1)$ such that for each $\varepsilon \in (0,\varepsilon_0)$, $a \in \mathbb R^{n}$ with $|a|r_{\varepsilon} \leq 1$ and for all $\mathcal{V}_{j} \in C^{2,\alpha}_{\mu;d}(B_{r_{\varepsilon}}(0)\backslash \{0\}), j=1,2$, with $\|\mathcal{V}_{j}\|_{(2,\alpha),\mu,r_{\varepsilon};d} \leq cr_{\varepsilon}^{\frac{49}{25}+d_n-\mu - \frac{n}{2}}$ we have that $Q^{i}_{\varepsilon,R,a}$ satisfies the following inequalities
\begin{equation}\label{eq025}
\begin{array}{l}
\|Q_{\varepsilon,R,a}^{i}(h+\mathcal{V}_1) - Q^{i}_{\varepsilon,R,a}(h+\mathcal{V}_2)\|_{(0,\alpha),\mu-2,r_{\varepsilon};d} \\
\leq Cr_{\varepsilon}^{\frac{n-2}{2}}\|\mathcal{V}_1 - \mathcal{V}_2\|_{(2,\alpha),\mu,r_{\varepsilon};d} \left(r_\varepsilon^\mu \|\mathcal{V}_1\|_{(2,\alpha),\mu,r_{\varepsilon};d} +r_\varepsilon^\mu \|\mathcal{V}_2\|_{(2,\alpha),\mu,r_{\varepsilon};d}+Cr^{4}\right)
\end{array}
\end{equation}
	and
\begin{equation}\label{eq026}
\| Q_{\varepsilon,R,a}(h)\|_{(0,\alpha),\mu-2,r_{\varepsilon};d}\leq Cr^{\frac{n+14}{2}-\mu}_\varepsilon,
\end{equation}
	where the constant $C> 0$ does not depend on $\varepsilon$, $R$ and $a$.
\end{lemma}

\begin{proof}
	First, note that by \eqref{resto}  we can write 
$$\begin{array}{rl}
Q^{i}_{\varepsilon,R,a}(h+\mathcal{V}_1) - Q^{i}_{\varepsilon,R,a}(h+\mathcal{V}_{0}) =
\end{array}$$
$$\begin{array}{rl}
& =  \displaystyle n\int_{0}^{1} \int_{0}^{1} \left(|\mathcal U_{\varepsilon,R,a}+s\mathcal Z_t|^{\frac{12-4n}{n-2}}\langle \mathcal U_{\varepsilon,R,a}+s\mathcal Z_t,\mathcal Z_t\rangle\right.\\
\\
&   +|\mathcal U_{\varepsilon,R,a}+s\mathcal Z_t|^{\frac{8-2n}{n-2}}\left(\langle\mathcal Z_t,\mathcal V_1-\mathcal V_0 \rangle(u_{\varepsilon,R,a}\Lambda_i+sz_{it})\right.\\
\\
&  \left.\left.+\langle \mathcal U_{\varepsilon,R,a}+s\mathcal Z_t,\mathcal V_1-\mathcal V_0\rangle z_{it}\right)\right)dsdt,
\end{array}$$
where $\mathcal Z_{t} =  h+t\mathcal V_{1} + (1-t)\mathcal V_0=(z_{1t},\ldots,z_{dt})$. From this we obtain
$$\|Q^{i}_{\varepsilon,R,a}(h+\mathcal{V}_1) - Q^{i}_{\varepsilon,R,a}(h+ \mathcal{V}_{2})\|_{(0,\alpha),[\sigma, 2\sigma]}\leq $$
$$\begin{array}{rcl}
  & \leq  & C \left(\|\mathcal V_1\|_{(0,\alpha),[\alpha,2\alpha]}+\|\mathcal V_0\|_{(0,\alpha),[\alpha,2\alpha]}+\|h\|_{(0,\alpha),[\alpha,2\alpha]}\right)\times\\
\\
& &  \times\left\|\mathcal V_{1} - \mathcal V_{0}\right\|_{(0,\alpha),[\sigma, 2\sigma]}\max_{0\leq s,t\leq 1} \left\|\left|\mathcal{U}_{\varepsilon,R,a} + s\mathcal Z_t\right|^{\frac{6-n}{n-2}}\right\|_{(0,\alpha),[\sigma,2\sigma]}.
\end{array}$$	
Now, note that 
$	\|\mathcal{V}_{j}\|_{(2,\alpha),\mu,r_{\varepsilon};d} \leq cr_{\varepsilon}^{\frac{49}{25}+d_n-\mu - \frac{n}{2}}
$
	implies 
$|v_{ij}(x)| \leq cr_{\varepsilon}^{\frac{49}{25}+d_n-\frac{n}{2} }
$
	for all $x \in B_{r_{\varepsilon}}(0)\backslash \{0\}$. Using $\eqref{2.6}$, yields
	\begin{eqnarray*}
		u_{\varepsilon,R,a}\Lambda_{i}(x)  + v_{ij}(x) +h_i&\geq& C_{1}\Lambda_{i}\varepsilon|x|^{\frac{2-n}{2}} - cr_{\varepsilon}^{\frac{49}{25}+ d_n - \frac{n}{2} } -c|x|^4\\
		& = & \varepsilon|x|^{\frac{2-n}{2}}\left(C_1 \Lambda_i - c(|x|r_{\varepsilon}^{-1})^{\frac{n-2}{2}}\varepsilon^{s(d_n+24/25) - 1}\right.\\
		& &\left. -c\varepsilon^{-1}|x|^{\frac{n+6}{2}}\right)
	\end{eqnarray*}
	with $s(d_n+24/25) -1 >0$ and $\varepsilon^{-1}|x|^{\frac{n+6}{2}}\leq \varepsilon^\eta$ for some $\eta>0$, since $s > (d_n + 24/25)^{-1}$. Hence, it follows that 
	\begin{equation}\label{r001}
	C_{3}\varepsilon |x|^{\frac{2-n}{2}} \leq u_{\varepsilon,R,a}(x)\Lambda_{i} +h_i+  v_{ij}(x) \leq C_{4} |x|^{\frac{2-n}{2}},
	\end{equation}
	and consequently
	\begin{equation}\label{r002}
|(\mathcal{U}_{\varepsilon,R,a} +s  \mathcal{Z}_t)(x)|^{\frac{6-n}{n-2}}\leq c|x|^{\frac{n-6}{2}}, 
	\end{equation}
	for small enough $\varepsilon > 0$, since $|x|\leq r_{\varepsilon}$,
	for some positive constant $c$ independent of $\varepsilon$, $a$ and $R$. The estimate for the full H\"older norm is similar. Hence, we conclude that
	$$\max_{0\leq s,t\leq 1}\||\mathcal U_{\varepsilon,R,a}+s\mathcal Z_t|^{\frac{6-n}{n-2}}\|_{(0,\alpha),[\sigma,2\sigma]}\leq c\sigma^{\frac{n-6}{2}}.$$

	Therefore, it follows \eqref{eq025}. 	Now, note that
	$$\begin{array}{rl}
	Q^{i}_{\varepsilon,R,a}(h)  & =  \displaystyle n\int_{0}^{1} \int_{0}^{1} \left(|\mathcal U_{\varepsilon,R,a}+sth|^{\frac{12-4n}{n-2}}\langle \mathcal U_{\varepsilon,R,a}+sth,th\rangle\times\right.\\
	\\
	&  \times\langle \mathcal U_{\varepsilon,R,a}+sth,th\rangle(u_{\varepsilon,R,a}\Lambda_i+sth_i) \\
	\\
	&   +|\mathcal U_{\varepsilon,R,a}+sth|^{\frac{8-2n}{n-2}}\left(\langle th,h \rangle(u_{\varepsilon,R,a}\Lambda_i+sth_i)\right.\\
	\\
	&  \left.\left.+\langle \mathcal U_{\varepsilon,R,a}+sth,h\rangle th_i\right)\right)dsdt.
	\end{array}$$
	This implies that
	$$\begin{array}{rcl}
\|Q^{i}_{\varepsilon,R,a}(h)\|_{(0,\alpha),[\sigma,2\sigma]} & \leq & c_n\|h\|^2_{(0,\alpha),[\sigma,2\sigma]}\max_{0\leq s,t\leq 1}\||\mathcal U_{\varepsilon,R,a}+sth|^{\frac{6-n}{n-2}}\|_{(0,\alpha),[\sigma,2\sigma]}.
	\end{array}$$
	
Thus, using \eqref{r002} we obtain
		$$\begin{array}{rcl}
\sigma^{2-\mu}\|Q^{i}_{\varepsilon,R,a}(h)\|_{(0,\alpha),[\sigma,2\sigma]} & \leq & c_n\sigma^{\frac{n+14}{2}-\mu}.
	\end{array}$$
	Therefore, we obtain \eqref{eq026}.	\qedhere
	
\end{proof}



Now, consider the map
\begin{equation*}
\mathcal{N}_{\varepsilon}(R,a,\phi,h,\cdot): \mathcal{B}_{\varepsilon,\tau} \rightarrow C^{2,\alpha}_{\mu;d}(B_{r_{\varepsilon}}(0)\backslash \{0\})
\end{equation*}
defined by $\mathcal{N}_{\varepsilon}(R,a,\phi,h,v) = G_{\varepsilon,R,r_{\varepsilon},a}(f_{1},\ldots, f_{d})$, where $\tau>0$ is a constant, $\mathcal{B}_{\varepsilon,\tau}$ is the ball of radius $\tau r_{\varepsilon}^{\frac{49}{25}+d_n-\mu - \frac{n}{2}}$ in $C^{2,\alpha}_{\mu;d}(B_{r_{\varepsilon}}(0)\backslash \{0\})$  and $f_{i}$ is the right hand side of the system \eqref{2.4},
for suitable parameters $\varepsilon$, $R$, $a$ and $\phi=(\phi_1,\ldots,\phi_d)$.

\begin{lemma}\label{lem004} 
Let $\mu\in(1,3/2)$. Given a constant $\kappa>0$,  there exists a constant $\varepsilon_{0}>0$ such that, for each $\varepsilon \in (0,\varepsilon_{0})$  the map $\mathcal{N}_{\varepsilon}(R,a,\phi,h,\cdot)$ is well defined in $\mathcal{B}_{\varepsilon,\tau}$ for $\phi_i \in C_2^{2,\alpha}(\mathbb{S}^{n-1}_{r_{\varepsilon}})^\perp$, $i=1,\ldots,d$, with $\|\phi\|_{(2,\alpha),r_{\varepsilon};d} \leq kr_{\varepsilon}^{\frac{49}{25}+d_n-\frac{n}{2}}$.
\end{lemma}

\begin{proof}
We need only to show that the right hand side in \eqref{2.4} belongs to the domain of the right inverse $G_{\varepsilon,R,r_{\varepsilon},a}$ which is the space $C_{\mu-2;d}^{0,\alpha}(B_r(0)\backslash\{0\})$. Indeed, by \ref{H1} we get that
\begin{equation*}
\beta^{-\frac{4}{n-2}}\sum_{j=1}^{2}\left(A_{ij}- c_{n}R_{g_{0}}\delta_{ij}\right)(u_{\varepsilon,R,a}\Lambda_{j} + h+v_{\phi_j}+ v_{j})(x) = O(|x|^{-1/2})=O(|x|^{\mu-2}).
\end{equation*}
The estimates of the remain terms will follow using the Lemma \ref{re} and argument similar to the Lemma 3.3 in \cite{almir}. \qedhere
\end{proof}

Since the map $\mathcal{N}_{\varepsilon}(R,a,\phi,h,\cdot)$ is well-defined in $\mathcal{B}_{\varepsilon,\tau}$ we can reduce the problem of finding a solution to the system \eqref{2.4}, to the problem of finding a fixed point for $\mathcal N_{\varepsilon}(R,a,\phi,h,\cdot)$. This is the content of the next result

\begin{theorem}\label{teo002} Let $\mu \in (1,5/4)$, $k>0$ and $\tau > 0$.  There exists  $\varepsilon_{0}\in (0,1)$ such that for each $\varepsilon \in (0,\varepsilon_{0}]$, $|b|\leq 1/2$, $a \in \R^{n}$ with $|a|r_{\varepsilon}^{\frac{23}{24}} \leq 1$, and $\phi \in C_2^{2,\alpha}(\mathbb{S}^{n-1}_{r_{\varepsilon}}))^\perp$ with $\|\phi\|_{(2,\alpha),r_{\varepsilon}}\leq kr_{\varepsilon}^{\frac{49}{25}+d_n-\frac{n}{2}}$, there exists a solution $\mathcal{V}_{\varepsilon,R,a,\phi}$ of the problem
\begin{equation}\label{eq001}
\left\{
\begin{array}{lc}
H_{g_0}(\beta^{-1}(\mathcal{U}_{\varepsilon,R,a} +h+ \mathcal{V}_{\phi} + \mathcal{V}_{\varepsilon,R,a,\phi})) =0 \quad \mbox{in} \quad B_{r_{\varepsilon}}(0)\backslash \{0\}\\
\pi_{r_{\varepsilon}}''((\mathcal{V}_{\phi} + \mathcal{V}_{\varepsilon,R,a,\phi})|_{\partial B_{r_{\varepsilon}}(0)}) = \phi \quad \mbox{on} \quad \partial B_{r
_{\varepsilon}}(0).
\end{array}
\right.
\end{equation}
Moreover,
\begin{equation*}
\|\mathcal{V}_{\varepsilon,R,a,\phi}\|_{(2,\alpha),\mu,r_{\varepsilon};d} \leq \tau r_{\varepsilon}^{2+d_n-\mu-\frac{n}{2}}
\end{equation*}
and 
\begin{equation}\label{eq051}
	\|\mathcal{V}_{\varepsilon,R,a,\phi_{1}} - \mathcal{V}_{\varepsilon,R,a,\phi_{2}}\|_{(2,\alpha),\mu,r_{\varepsilon};d} \leq Cr_{\varepsilon}^{\delta-\mu}\|\phi_{1} - \phi_{2}\|_{(2,\alpha),r_{\varepsilon};d}
\end{equation}
for some constant $\delta >0$ which does not depend on $\varepsilon, R, a$, $h$ and $\phi_{j} = (\phi_{1j},\ldots,\phi_{dj})$, $j=1,2$.
\end{theorem}

\begin{proof} 
In order to prove the existence of the solution to the problem we need to prove that the map $\mathcal{N}_{\varepsilon}(R,a,\phi,h,\cdot)$ has a fixed point. But, to do that it is enough to show that 
\begin{equation*}
\|\mathcal{N}_{\varepsilon}(R,a,\phi,h,0)\|_{(2,\alpha),\mu,r_{\varepsilon};d} < \frac{1}{2}\tau r_{\varepsilon}^{2+d_n-\mu-\frac{n}{2}}
\end{equation*}
and
$$
\|\mathcal{N}_{\varepsilon}(R,a,\phi,h,\mathcal{V}_{1}) - \mathcal{N}_{\varepsilon}(R,a,\phi,h,\mathcal{V}_{2})\|_{(2,\alpha),\mu,r_{\varepsilon};d} < \frac{1}{2}\|\mathcal{V}_{1} - \mathcal{V}_{2}\|_{(2,\alpha),\mu,r_{\varepsilon};d}
$$
for all $\mathcal{V}_{i} \in \mathcal{B}_{\varepsilon,\tau}$,  $i=1,2$.  

Using the Corollary \ref{inver}, Lemma \ref{re}, \eqref{po}, \eqref{2.6}, the hypothese \ref{H1} and a similiar argument as in Theorem 3.8 in \cite{almir} we get the result.

To show \eqref{eq051} we use the fact that the solution is a fixed point, the previous estimates and that 
$$\|\mathcal{V}_{\varepsilon,R,a,\phi_{1}} - \mathcal{V}_{\varepsilon,R,a,\phi_{2}}\| \leq 2\|\mathcal{N}_{\varepsilon}(R,a,\phi_1,h,\mathcal{V}_{\varepsilon,R,a,\phi_{2}}) -\mathcal{N}_{\varepsilon}(R,a,\phi_2,h,\mathcal{V}_{\varepsilon,R,a,\phi_{2}})\|. $$
\qedhere

\end{proof}

\subsection{Exterior Analysis}\label{sec003}

Let $(M,g_{0})$ be a nondegenerate compact Riemannian manifold with constant scalar curvature $R_{g_{0}} = n(n-1)$. Our main purpose in this section is to find a family of solutions to the system \eqref{S} in the complement of a ball centered at a fixed point $p\in M$, with suitable radius, in the manifold $M$. 

Since the difference between the potential $A$ and the scalar curvature is controlled only near the point $p$, this hypotheses will not help us to study the problem in the complement of a ball. Consequently, we will assume that $\Lambda = \left(\Lambda_1,\ldots,\Lambda_d\right)$ is a trivial solution of the system \eqref{S}, that is, $H_{g_{0}}(\Lambda) = 0$.

Let $r_{1} \in (0,1)$ and $\Psi : B_{r_{1}}(0) \rightarrow M$ be a normal coordinate system with respect to the metric $g = \beta^{\frac{4}{n-2}}g_{0}$ on $M$ centered in $p$ satisfying $\beta = 1 + O(|x|^{2})$ in $g$-normal coordinates. Define $G_{p} \in C^{\infty}(M\backslash \{p\})$ by $G_p\circ\Psi=\eta|x|^{2-n}$ in $B_{r_1}(p)$ and equal to zero in $M_{r_1}$, where $\eta$ is a smooth radial function equal to 1 in $B_{3r}(p)$, equal to zero in $\mathbb R^n\backslash B_{4r}(p)$ and with the estimates $|\nabla\eta(x)|\leq c|x|^{-1}$ and $|\nabla^2\eta(x)|\leq c|x|^{-2}$ for all $x\in B_{r_1}(p)$.

Our goal in this section is to solve the system 
\begin{equation}\label{S2}
H_{g_{0}}(\Lambda + \boldsymbol{\mathcal {G}}_{p,\rho} + \mathcal{U}) = 0  \quad \mbox{ in } \quad M_{r}=M\backslash B_{r}(p),
\end{equation}
for suitable parameter $r > 0$, where $\mathcal {G}_{p,\rho} = (\rho_1,\ldots,\rho_d)G_{p}$ and $\rho=(\rho_1,\ldots,\rho_d) \in \R^d$. Remember that $H_{g_0}$ is defined in Section \ref{sec001} and by assumption $\Lambda$ satisfies $H_{g_0}(\Lambda)=0$.

Following the strategy of the previous section, we will use the right inverse obtained in Section \ref{sec006} to reduce the problem of finding a solution to \eqref{S2}, to a fixed point problem. The function $\boldsymbol{\mathcal {G}}_{p,\rho}$ introduced in \eqref{S2} will be useful since the parameter $\rho$ will be necessary to match the Cauchy data in Section \ref{sec005}, besides this function is harmonic in $M_r$ and has an appropriate decay.


\subsubsection{Fixed point argument}\label{sec010}
 For a fixed $\varphi \in C^{2,\alpha}(\mathbb{S}^{n-1}_{r})$, we can consider the Poisson operator  $\mathcal{Q}_{r}(\varphi)$ associated with the laplacian in $\R^{n}\backslash B_{r}(0)$, which is the only solution of the problem
\begin{equation}\label{sol1}
\left\{
\begin{array}{lcl}
\Delta \mathcal{Q}_{r}(\varphi)= 0 &\mbox{ in } & \quad \R^{n}\backslash B_{r}(0) \\
Q_{r}(\phi) =  \varphi& \mbox{ on } & \quad \partial B_{r}(0),
\end{array}
\right.
\end{equation}
which satisfies 
\begin{equation*}
	\|Q_{r}(\varphi)\|_{C^{2,\alpha}_{1-n}(\R^{n}\backslash B_{r}(0))} \leq c r^{n-1}\|\varphi\|_{(2,\alpha),r}
\end{equation*}
where $c > 0$ is a constant that does not depend on $r$. For details, we refer \cite[Proposition 1.7.3]{almir}. It is well known that if $\varphi=\displaystyle\sum_{j=1}^\infty \varphi_j$, where $\varphi_j$ belongs to the eigenspace associated with the eigenvalue $i(i+n-2)$, then
$$\mathcal Q_r(\varphi)(x)=\sum_{j=1}^\infty r^{n+j-2}|x|^{2-n-j}\varphi_j.$$

For each $\varphi \in C^{2,\alpha}(\mathbb{S}^{n-1}_{r})$ which is $L^{2}$-orthogonal to the constant functions, let $u_{\varphi} \in C^{2,\alpha}_{\nu}(M_{r})$ such that $u_{\varphi} \equiv 0$ in $M_{r_{1}}$ and $u_{\varphi}\circ \Psi = \eta\mathcal{Q}_{r}(\varphi)$, where $\eta$ is a smooth, radial function equal to $1$ in $B_{r_{1}/2}(0)$ and vanishings $\R^{n}\backslash B_{r_{1}}(0)$, with $|x||\partial_{r}\eta(x)|\leq c$ and $|x|^2|\partial_{r}^{2}\eta(x)| \leq c$. Using the properties of the Poisson operator and of the cut function $\eta$ we can verify that 
\begin{equation}\label{eq052}
	\|u_{\varphi}\|_{C^{2,\alpha}_{\nu}(M_{r})} \leq cr^{-\mu}\|\varphi\|_{(2,\alpha),r}.
\end{equation}

Now using that $H_{g_0}(\Lambda)=0$ and the right inverse given by Proposition \ref{propo004}, we linearize the operator $H_{g_{0}}$ at $\Lambda$ and find that to solve \eqref{S2} with $\mathcal U$ replaced by $\mathcal{U}_{\varphi} + \mathcal{V}$, where $\mathcal{U}_{\varphi} = (u_{\varphi_{1}},\ldots, u_{\varphi_{d}})$, is equivalente to find a fixed point for the map  $\mathcal{M}_{r}(\rho,\varphi,\cdot): C^{2,\alpha}_{\nu;d}(M_{r})\rightarrow C^{2,\alpha}_{\nu;d}(M_{r})$ given by $\mathcal{M}_{r}(\rho,\varphi,\mathcal{V}) = - G_{r,g_{0}}(h_{1},\ldots,h_{d})$, where 
\begin{equation}\label{eq050}
h_{i} = Q^{i}(\boldsymbol{\mathcal G}_{p,\rho} + \mathcal{U}_{\varphi} + \mathcal{V}) +\mathcal L_{g_{0}}^{i}(\boldsymbol{\mathcal G}_{p,\rho}+\mathcal{U}_{\varphi})
\end{equation}
$i=1,\ldots,d$, for suitable parameters $\rho \in \R^d$ and $\varphi_{i} \in C^{2,\alpha}(\mathbb{S}^{n-1}_{r})$, where $\mathcal L_{g_0}^i$ is defined in \eqref{li} and  $Q^i$ is defined in \eqref{resto} with $\mathcal U_0$ replaced by $\Lambda$. This is the content of the next result, which the proof is similar to the Theorem \ref{teo002} and using that the operator $G_{r,g_0}$ is bounded by Proposition \ref{propo004}.

\begin{theorem}\label{teo001} Let $\nu \in (3/2-n, 2-n)$, $\gamma>0$ and $\zeta >0$ fixed constants. There exists $r_{2}>0$ such that if $r \in (0,r_{2})$, $\rho \in \R^d$ with $|\rho_i|^{2} \leq r^{d_n-\frac{51}{25}+\frac{3n}{2}}$, and for each $i=1,\ldots,d$ the function $\varphi_{i} \in C^{2,\alpha}(\mathbb{S}^{n-1}_{r})$ is $L^{2}$-orthogonal to the constant functions with $\|\varphi_{i}\|_{(2,\alpha),r}\leq \zeta r^{\frac{49}{25}+d_n-\frac{n}{2}}$, then there is a solution $\mathcal{V}_{\rho,\varphi}$ of the system 
\begin{equation*}
\left\{
\begin{array}{lc}
H_{g_{0}}^{i}(\Lambda+ \boldsymbol{\mathcal G}_{p,\rho} + \mathcal{U}_{\varphi} + \mathcal{V}_{\rho,\varphi}) = 0 \quad \mbox{in} \quad M_{r}\\
(\mathcal{U}_{\varphi} + \mathcal{V}_{\rho,\phi})\circ \Psi|_{\partial B_{r}(0)} - \varphi \in \R^{2} \quad \mbox{on} \quad \partial M_{r}.
\end{array}
\right.
\end{equation*}	

Moreover, 
\begin{equation*}
	\|\mathcal{V}_{\rho,\varphi}\|_{C^{2,\alpha}_{\nu;d}(M_{r})} \leq \gamma r^{2+d_n-\nu-\frac{n}{2}}
\end{equation*}
and 
\begin{equation*}
	\|\mathcal{V}_{\rho,\varphi_{1}} - \mathcal{V}_{\rho,\varphi_{2}}\|_{C^{2,\alpha}_{\nu;d}(M_{r})} \leq Cr_{\varepsilon}^{\delta-\mu}\|\varphi_{1} - \varphi_{2}\|_{(2,\alpha),r;d}
\end{equation*}
for some constant $\delta >0$ small enough independent of $r$, where $\varphi_{j} = (\varphi_{1j},\ldots,\varphi_{2j})$ for $j=1,2$.
\end{theorem}
\begin{proof}
The proof is similar to that of Theorem \ref{teo002}, which, besides the hypotheses,  we use Proposition \ref{propo004}, \eqref{eq052}, \eqref{eq050} and the fact that the supports of $\mathcal G_{p,\rho}$ and $\mathcal U_\varphi$ belong to the ball $B_{r_1}(p)$.
\end{proof}

\section{Cauchy data matching: Gluing method}\label{sec005}

In this section we will proof the existence of singular solution to the system \eqref{S}. Although the computations above are rather technical, the main idea is simple and it consists in finding the appropriate parameters in a way that the solutions constructed in the interior and exterior analysis coincide on the boundary up to order one. Thus, using elliptic regularity we get a smooth solution. In order to find the right parameters, we will again use a fixed point argument.  More precisely, we will prove the following theorem
\begin{theorem}[Theorem \ref{teo004}]\label{teo003}
Let $(M,g_0)$ be a closed Riemannian manifold with dimension \textcolor{red}{$3\leq n\leq 5$} and constant scalar curvature $n(n - 1)$.  Assume that the metric is nondegenerate at some $\Lambda\in\mathbb S^{d-1}_+$ (see Definition \ref{def001}). Suppose that the potential $A$ satisfies the hypotheses \eqref{H1}. Then there exists a constant $\varepsilon_0>0$ and a one-parameter family of positive smooth functions $\mathcal V_\varepsilon=(v_{1,\varepsilon},\ldots,v_{d,\varepsilon})$ on $M\backslash\{p\}$ defined for $\varepsilon\in(0,\varepsilon_0)$ such that 
\begin{enumerate}
    \item \label{cond001} each $\mathcal V_\varepsilon$ is a smooth solution to the system
  $$\Delta_{g}v_{i,\varepsilon}- \sum_{j=1}^{d}A_{ij}(x)v_{j,\varepsilon} + \frac{n(n-2)}{4}|\mathcal{V}_{\varepsilon}|^{\frac{4}{n-2}}v_{i,\varepsilon} = 0, \;\;\; \mbox{ in }M\backslash\{p\},$$
  for all $i=1,\ldots,d,$;
    \item \label{cond002} near the singularity $p$, $\mathcal V_\varepsilon$ is asymptotically to some Fowler-type solution $\mathcal U_{\varepsilon,R,a}=u_{\varepsilon,R,a}\Lambda $; and
\item\label{cond003} $\mathcal V_\varepsilon\rightarrow\Lambda$ as $\varepsilon\rightarrow 0$.
\end{enumerate}
\end{theorem}

By Theorem \ref{teo002}, for sufficiently small $\varepsilon>0$, there is a family of positive functions, given by $\mathcal A_\varepsilon(R,a,\phi,h)=(\mathcal  A_\varepsilon(R,a,\phi,h)_1,\ldots,\mathcal  A_\varepsilon(R,a,\phi,h)_d),$ where 
$$\mathcal A_\varepsilon(R,a,\phi,h)_i= u_{\varepsilon,R,a}\Lambda_i+h_i+v_{\phi_i}+\mathcal U_{\varepsilon,R,a,\phi,i}$$
such that 
$$\left\{
\begin{array}{lc}
H_{g_0}(\beta^{-1}\mathcal  A_\varepsilon(R,a,\phi,h)) =0 \quad \mbox{in} \quad B_{r_{\varepsilon}}(p)\backslash \{p\}\\
\pi_{r_{\varepsilon}}''(A_\varepsilon(R,a,\phi,h)) = \phi \quad \mbox{on} \quad \partial B_{r
	_{\varepsilon}}(p).
\end{array}
\right.$$
Here the function $h=(h_1,\ldots,h_d)$ is defined in \eqref{eq031}.

Also, by Theorem \ref{teo001}, for sufficiently small $\varepsilon>0$, there is a family of positive functions $\mathcal B_\varepsilon(\rho,\varphi)=(\mathcal B_\varepsilon(\rho,\varphi)_1,\ldots,\mathcal B_\varepsilon(\rho,\varphi)_d)$, given by
$$\mathcal B_\varepsilon(\rho,\varphi)_i=\Lambda_i+ \rho_iG_p + u_{\varphi_i} + \mathcal{V}_{\rho,\varphi,i}$$
such that
$$\left\{
\begin{array}{rll}
H_{g_{0}}(\mathcal B_\varepsilon(\rho,\varphi)) = 0 & \quad \mbox{in} & \quad M_{r}\\
\mathcal B_\varepsilon(\rho,\varphi)\circ \Psi|_{\partial B_{r}(0)} - \varphi \in \R^{2} & \quad \mbox{on} & \quad \partial M_{r}.
\end{array}
\right.
$$
Remember that in $\partial M_{r_\varepsilon}$ we have $G_p=|x|^{2-n}$.


Our main purpose is to show the existence of parameters $(\rho,R)\in\mathbb R^d\times \mathbb R_+$, $a\in\mathbb R^n$, a function $h$ such as in \eqref{eq031} and $\phi,\varphi\in C^{2,\alpha}(\mathbb S_{r_\varepsilon}^{n-1})^{d}$ such that 
\begin{equation}\label{eq009}
\left\{\begin{array}{lcl}
\mathcal A_\varepsilon(R,a,\phi,h) & = & \beta\mathcal B_{r_\varepsilon}(\rho,\varphi)\\
\partial_r( \mathcal A_\varepsilon(R,a,\phi,h)) & = & \partial_r(\beta\mathcal B_{r_\varepsilon}(\rho,\varphi))
\end{array}\right.,
\end{equation}
on $\partial B_{r_\varepsilon}(p)$, where $\beta=1+f$ with $f=O(|x|^2)$. 

Since we can take the function $\phi$ only in the high frequency space, see \eqref{eq054}, and the inverse map that we obtained in Corollary \ref{inver} gives us a function whose components in the high frequency spaces vanish, in order to solve this system we need to project $\eqref{eq009}$ separately in the low and high frequencies spaces. The function $h$ introduced in \eqref{eq031} will be important to solve $\eqref{eq009}$ and one of the main differences between solving the system \eqref{S} and its scalar case. 

Let $\omega_i,\vartheta_i\in C^{2,\alpha}(\mathbb S_{r_\varepsilon}^{n-1})$ such that
\[\|\omega_{i}\|_{(2,\alpha), r_{\varepsilon}}, \|\vartheta_{i}\|_{(2,\alpha), r_{\varepsilon}} \leq r_\varepsilon^{\frac{49}{25}+d_n-\frac{n}{2}}, \qquad i=1,\ldots, d \] 
where $d_{n}$ is given by Remark \ref{rem001}, $\omega_i$ belongs to the space spanned by the coordinate functions and $\vartheta_i$ is in the high frequency space. By Theorem \ref{teo001} we see that $\mathcal B_{\varepsilon}(\rho,\omega+\vartheta)$ is well defined, where we take $\zeta=2$. Here $\omega=(\omega_1,\ldots,\omega_d)$ and $\vartheta:=(\vartheta_1,\ldots,\vartheta_d)$. Now, let
$$\begin{array}{rcl}
     \phi_\vartheta^i & = & \pi_{r_\varepsilon}''(\left.(\beta\mathcal B^i_{r_\varepsilon}(\rho,\omega+\vartheta)-u_{\varepsilon,R,a}\Lambda_i)\right|_{\mathbb S_{r_\varepsilon}^{n-1}})\\
     & = & \pi_{r_\varepsilon}''\left(\left.(\Lambda_i  f+\rho_i f G_p+ f u_{\omega_i+\vartheta_i}+ f\mathcal V_{\rho,\omega+\vartheta,i}-u_{\varepsilon,R,a}\Lambda_i)\right|_{\mathbb S_{r_\varepsilon}^{n-1}}\right)+\vartheta_i.
\end{array}$$
By the estimates that we obtained in the Sections \ref{sec002} and \ref{sec003}, we get
$$\|\phi^i_\vartheta\|_{(2,\alpha),r_\varepsilon;d}\leq Cr_\varepsilon^{\frac{49}{25}+d_n-\frac{n}{2}},$$
for some positive constant $C$ that does not depend on $\varepsilon$. Thus, by Theorem \ref{teo002}, we get that $\mathcal A_\varepsilon(R,a,\phi_\vartheta,h)$ is well defined, which implies that
$$\pi_{r_\varepsilon}''(\mathcal A_{r_\varepsilon}(R,a,\phi_\vartheta,h))=\pi_{r_\varepsilon}''(\beta\mathcal  B_\varepsilon(\rho,\omega+\vartheta)).$$

By projecting the second equation of \eqref{eq009} in the high frequency space, we get 
\[r_\varepsilon\partial_r(v_{\vartheta_i}-u_{\vartheta_i})+\mathcal S^i_\varepsilon(a,b,\rho,\omega,\vartheta)=0,\]
on $\partial_rB_{r_\varepsilon}(0)$, where
$$\begin{array}{rcl}
     S^i_\varepsilon(a,b,\rho,\omega,\vartheta) & = & r_\varepsilon\partial_rv_{\phi_{\vartheta_i}-\vartheta_i}+r_\varepsilon\partial_r\pi''_{r_\varepsilon}(u_{\varepsilon,R,a})\Lambda_{i}-r_\varepsilon\partial_r\pi''_{r_\varepsilon}(\beta\mathcal V_{\rho,\omega+\vartheta,i})\\
     \\
     & & +r_\varepsilon\partial_r\pi''_{r_\varepsilon}(\mathcal U_{\varepsilon,R,a,\phi_\vartheta,i}- f-\rho_i \beta G_p- fu_{\omega_i+\vartheta_i}).
\end{array}$$

Now, we note that the map $\mathcal P:\pi''(C^{2,\alpha}(\mathbb S^{n-1}))\rightarrow\pi''(C^{1,\alpha}(\mathbb S^{n-1}))$ defined as 
$$\mathcal P(\phi):=r_\varepsilon\partial_r(v_{\phi_{r_\varepsilon}}-u_{\phi_{r_\varepsilon}})(r_\varepsilon\cdot)$$
is an isomorphism, where $\phi_{r_\varepsilon}(x)=\phi(r_\varepsilon^{-1} x)$ (See \cite{jleli,MR1763040}). Remember that $v_\phi$ is given by \eqref{eq054} and $u_\phi$ is defined in the Section \ref{sec010}.

Consider the map $\mathcal H_\varepsilon(a,b,\rho,\omega,\cdot):\mathcal D_\varepsilon\rightarrow\pi''(C^{2,\alpha}(\mathbb S^{n-1}))^d$
given by
$$\mathcal H_\varepsilon(a,b,\rho,\omega,\vartheta)=-\mathcal Z^{-1}(\mathcal S_\varepsilon(a,b,\rho,\omega,\vartheta_{r_\varepsilon})(r_\varepsilon\cdot)),$$
where $\mathcal Z:\mathcal \pi''(C^{2,\alpha}(\mathbb S^{n-1}))^d\rightarrow\pi''(C^{1,\alpha}(\mathbb S^{n-1}))^d$ is the isomorphism given by
$\mathcal Z(\phi_1,\ldots,\phi_d)=(\mathcal P(\phi_1),\ldots,\mathcal P(\phi_d))$, and
$$\mathcal D_\varepsilon:=\left\{(\vartheta_1,\ldots,\vartheta_d)\in  \pi''(C^{2,\alpha}(\mathbb S^{n-1}))^d; \|\vartheta_i\|\leq r_\varepsilon^{\frac{49}{25}+d_n-\frac{n}{2}}\right\}.$$

By the estimates obtained in Theorems  \ref{teo002} and \ref{teo001}, we get that $\mathcal S_\varepsilon(a,b,\rho,\omega,\vartheta_{r_\varepsilon})=O(r_\varepsilon^{2+d_n-\frac{n}{2}})$ and therefore we obtain the following lemma.

\begin{lemma}\label{lem001}
There exists a constant $\varepsilon_0>0$ such that if $\varepsilon\in(0,\varepsilon_0)$, $a\in\mathbb R^n$ with $|a|^2\leq r_\varepsilon^{d_n-\frac{n}{2}}$, $(b,\rho)\in\mathbb R\times\mathbb R^d$ with $|b|\leq 1/2$ and $|\rho_i|^2\leq r_\varepsilon^{d_n-2+\frac{3n}{2}}$, and $\omega_i\in C^{2,\alpha}(\mathbb S_{r_\varepsilon}^{n-1})$, for $i=1,\ldots,d$, belongs to the space spanned by the coordinate functions and with norm bounded by $r_\varepsilon^{\frac{49}{25}+d_n-\frac{n}{2}}$, then the map $\mathcal H_\varepsilon(a,b,\rho,\omega,\cdot)$ has a fixed point in $\mathcal D_\varepsilon$.
\end{lemma}

We denote this fixed point simply by $\vartheta=(\vartheta_1,\ldots, \vartheta_d)$, which depends continuously on $\varepsilon$, $a$, $b$, $\rho$ and $\omega$. By Proposition \ref{propo002}, in the boundary $\partial B_{r_\varepsilon}(p)$, we get that
\begin{equation}\label{eq055}
    \begin{array}{rcl}
\mathcal A_\varepsilon(R,a,\phi_{\vartheta},h)_i & = &\displaystyle \Lambda_i+b\Lambda_i+\frac{\varepsilon^2}{4(1+b)}\Lambda_ir_\varepsilon^{2-n}+ \mathcal U_{\varepsilon,R,a,\phi_{\vartheta}}\\
\\
& & +((n-2)u_{\varepsilon,R}
+r_\varepsilon\partial_ru_{\varepsilon,R}(r_\varepsilon\theta))\Lambda_ir_\varepsilon a\cdot\theta\\
\\
& & +h_i+v_{\phi^i_{\vartheta}} +O(|a|^2r_\varepsilon^2)+O(\varepsilon^{2\frac{n+2}{n-2}}r_\varepsilon^{-n}),
\end{array}
\end{equation}
where $h=(h_1,\ldots,h_{d-1},0)$ is defined in \eqref{eq031}. If $|x|=r$ we have that $h(x)=(\eta,0)+r_\varepsilon(\langle A_1,\theta\rangle,\ldots,\langle A_{d-1},\theta\rangle,0 )$ and $\partial_r h(x)=(\langle A_1,\theta\rangle,\ldots,\langle A_{d-1},\theta\rangle,0 )$, with $\eta=(\eta_1,\ldots,\eta_{d-1})\in\mathbb R^{d-1}$ and $A_i\in \mathbb R^n$, for all $i=1,\ldots,d-1$. For $i=1,\ldots,d-1$, we set $\eta_i=-b\Lambda_i+b_i\Lambda_i$, with $|b_2|\leq 1/2$. Let us call $b_d:=b$.

In the exterior manifold $M_r$, in conformal normal coordinates system in the neighborhood of $\partial M_{r_\varepsilon}$, namely $\Omega_{r_\varepsilon,\frac{1}{2}r_1}$, using that $\beta=1+f$, we have
\begin{equation}\label{eq056}
    \begin{array}{rcl}
\beta\mathcal B_{r_\varepsilon}(\rho,\omega+\vartheta)_i & = & \Lambda_i+\rho_i r_\varepsilon^{2-n}+u_{\omega_i+\vartheta_i}+ f\Lambda_i+O(|\rho|r_\varepsilon^{3-n})\\
\\
& & +fu_{\omega_i+\vartheta_i}+f\mathcal V_{\rho,\omega+\vartheta,i}.
\end{array}
\end{equation}

To solve the projected part of the system \eqref{eq009} in the low frequency space, we will now project \eqref{eq009} in the direction of the constant functions and in the direction of the coordinate functions. By projecting the system \eqref{eq009} on the set of functions spanned by the constant function, this yields the equations

\begin{equation}\label{eq024}
\left\{\begin{array}{rcl}
\displaystyle b_i\Lambda_i+\left(\frac{\varepsilon^2}{4(1+b)}\Lambda_i-\rho_i\right)r_\varepsilon^{2-n} & = & \mathcal H^{i,0}_{\varepsilon}(a,b,\rho,\omega)\\
\displaystyle(2-n)\left(\frac{\varepsilon^2}{4(1+b)}\Lambda_i-\rho_i\right)r_\varepsilon^{2-n} & =  & r_\varepsilon\partial_r\mathcal H^{i.0}_{\varepsilon}(a,b,\rho,\omega).
\end{array}\right.
\end{equation}
where $\mathcal H^{i,0}_{\varepsilon}(a,b,\rho,\omega)$ denotes the projection of the remaining terms of \eqref{eq055} and \eqref{eq056} in the space of the constant functions, which by the estimates obtained in the previous sections has the order $ O(r_\varepsilon^{2+d_n-\frac{n}{2}})$. It is not difficult to see that a fixed point of the map $\mathcal G_{\varepsilon,a,\omega}:\mathcal D_{0,\varepsilon}\rightarrow\mathbb R^{2d}$, given by $\mathcal G_{\varepsilon,a,\omega}(b,\rho)=(\mathcal G_1,\ldots,\mathcal G_d,\mathcal F_1,\ldots,\mathcal F_d)$, where
$$\begin{array}{rcl}
     \mathcal G_i & = & \displaystyle\frac{r_\varepsilon}{(n-2)\Lambda_i}\partial_r\mathcal H^{i,0}_{\varepsilon}(a,b,\rho,\omega)+\frac{1}{\Lambda_i}\mathcal H^{i,0}_{\varepsilon}(a,b,\rho,\omega)\\
     \\
   \mathcal F_i & = & \displaystyle \frac{\varepsilon^2}{4(1+b)}\Lambda_i+\frac{r_\varepsilon^{n-1}}{n-2}\partial_r\mathcal H^{i,0}_{\varepsilon}(a,b,\rho,\omega),
\end{array}$$
for each $i=1,\ldots, d$, and
$$\mathcal D_{0,\varepsilon}:=\left\{(b,\rho)\in\mathbb R^d\times\mathbb R^d;|b_i|\leq 1/2\;\mbox{ and }\;|\rho_i|^2\leq  r_\varepsilon^{d_n-\frac{51}{25}+\frac{3n}{2}}\right\},$$
is a solution to the system \eqref{eq024}.

Since $\mathcal H^{i,0}_{\varepsilon}(a,b,\rho,\omega) = O(r_\varepsilon^{2+d_n-\frac{n}{2}})$, we use the previous estimates to obtain the following lemma.

\begin{lemma}\label{lem002}
There exists a constant $\varepsilon_1>0$ such that if $\varepsilon\in(0,\varepsilon_1)$, $a\in\mathbb R^n$ with $|a|^2\leq r_\varepsilon^{d_n-\frac{n}{2}}$ and $\omega\in C^{2,\alpha}(\mathbb S_{r_\varepsilon}^{n-1})^d$ belongs to the space spanned by the coordinate functions and with norm bounded by $r_\varepsilon^{\frac{49}{25}+d_n-\frac{n}{2}}$, then the map $\mathcal G_{\varepsilon,a,\omega}$ has a fixed point in $\mathcal D_{0,\varepsilon}$ which depends continuously on the parameter $\varepsilon$, $a$ and $\omega$.
\end{lemma}

Finally, we project the system \eqref{eq009} over the space of functions spanned by the coordinate functions. We get
\begin{equation}\label{eq032}
    \left\{\begin{array}{rcl}
     F(r_\varepsilon)\Lambda_i r_\varepsilon a_j+r_\varepsilon A_{ij}-\omega_{ij} & = & \mathcal H^{i,j}_{\varepsilon}(a,\omega)\\
     G(r_\varepsilon)\Lambda_ir_\varepsilon a_j+r_\varepsilon A_{ij}-(1-n)\omega_{ij} &= & r_\varepsilon\partial_r\mathcal H^{i,j}_{\varepsilon}(a,\omega)
\end{array}\right.,
\end{equation}
$i=1,\ldots,d$ and $j=1,\ldots,n$, where  $\omega=(\omega_1,\ldots,\omega_d)$, $\omega_i=\displaystyle\sum_{k=1}^n\omega_{ik}e_k$,
$$\begin{array}{rcl}
     F(r_\varepsilon) & = & (n-2)u_{\varepsilon,R}(r_\varepsilon\theta)+r_\varepsilon\partial_ru_{\varepsilon,R}(r_\varepsilon\theta),\\
     \\
     G(r_\varepsilon) & = & (n-2)u_{\varepsilon,R}(r_\varepsilon\theta)+nr_\varepsilon\partial_ru_{\varepsilon,R}(r_\varepsilon\theta)+r_\varepsilon^2\partial_r^2u_{\varepsilon,R}(r_\varepsilon\theta),
\end{array}$$
 and $\mathcal H^{i,j}_{\varepsilon}(a,\omega)$ denotes the projection of the remaining terms of \eqref{eq055} and \eqref{eq056} in the space of the coordinate functions, which by the estimates obtained in the previous sections has the order $ O(r_\varepsilon^{2+d_n-\frac{n}{2}})$. Here $A_{i}=(A_{i1},\ldots,A_{in})\in\mathbb R^n$ appears when $i=1,\ldots,d-1$.

Using Proposition \ref{propo002} and that $R^{\frac{2-n}{2}}=2(1+b)\varepsilon^{-1}$, we obtain that $F(r_\varepsilon)$ and $G(r_\varepsilon)$ satisfy the estimate 
$(n-2)(1+b)+O(\varepsilon^{2-s(n-2)}),$ with $2-s(n-2)>0$ (see Remark \ref{rem001}). Now, we choose 
$$A_{ij}=-(n-2)(1+b)\Lambda_i (a_j-\alpha_{ij}).$$ 
for $i=1,\ldots,d-1$ and $j=1,\ldots,n$. Let us call $\alpha_{dj}:=a_j$.

With this choice we see that a solution of \eqref{eq032} is a fixed point of the map $\mathcal K_{j,\varepsilon}:\mathcal D_{j,\varepsilon}\rightarrow\mathbb R^{2d},$ given by
$ \mathcal K_{j,\varepsilon}(\alpha_{j},\omega_j)=(\mathcal K_1,\ldots,\mathcal K_d,\mathcal Y_1,\ldots,\mathcal Y_d)$, where
$$\begin{array}{rcl}
     \mathcal K_i & = & \displaystyle\dfrac{r_\varepsilon\partial_r\mathcal H^{i,j}_{\varepsilon}+(n-1)\mathcal H^{i,j}_{\varepsilon}}{n(n-2)(1+b)\Lambda_ir}+O(\varepsilon^{2-s(n-2)})\alpha_{dj}\\
     \\
     \mathcal Y_i & = & \displaystyle n^{-1}(r_\varepsilon\partial_r\mathcal H^{i,j}_{\varepsilon}-\mathcal H^{i,j}_{\varepsilon})+O(\varepsilon^{2-s(n-3)})\alpha_{dj}
\end{array}$$
for $i=1,\ldots, d$, with $2-s(n-2)>0$, defined in the subset
$$\mathcal D_{j,\varepsilon}:=\{(\alpha_j,\omega_j)\in\mathbb R^d\times\mathbb R^d;|\alpha_{ij}|^2\leq \alpha_n r_\varepsilon^{d_n-\frac{n}{2}},|\omega_{ij}|\leq \beta_n r_{\varepsilon}^{\frac{49}{25}+d_n-\frac{n}{2}}\},$$
for some positive constants $\alpha_n$ e $\beta_n$ which depend only on $n$. Therefore, by the previous estimates, we obtain the following lemma.
\begin{lemma}\label{lem003}
There is a constant $\varepsilon_2>0$ such that if $\varepsilon\in(0,\varepsilon)$ then, for each $j=1,\ldots,n$, the system \eqref{eq032} has a solution $(\alpha_j,\omega_j)\in\mathcal D_{j,\varepsilon}$.
\end{lemma}

Now we are ready to proof the main theorem of this work.

\begin{proof}[Proof of Theorem \ref{teo003}]
By Theorem \ref{teo002} we have a family of solution $\beta^{-1}\mathcal  A_\varepsilon(R,a,\phi,h)$ to the system \eqref{S} in $\overline{B_{r_\varepsilon}(p)}\subset M$, for small enough $\varepsilon>0$. Also, by Theorem \ref{teo003} we find a family of solution $\mathcal B_{r_\varepsilon}(\rho,\varphi)$ to the system \eqref{S} in $M\backslash B_{r_\varepsilon}(p)$, for small enough $\varepsilon>0$. From Lemmas \ref{lem001}, \ref{lem002} and \ref{lem003}, we find $\varepsilon_0>0$ such that for all $\varepsilon\in(0,\varepsilon_0)$ there are parameters $R$, $a$, $\phi$, $\rho$ and $\varphi$ for which the functions $\mathcal  A_\varepsilon(R,a,\phi,h)$ and  $\mathcal B_{r_\varepsilon}(\rho,\varphi)$ coincide up to order one in $\partial B_{r_\varepsilon}(p)$. Hence using elliptic regularity it follows that the function $\mathcal V_\varepsilon$ defined  by $\mathcal V_\varepsilon=\mathcal  A_\varepsilon(R,a,\phi,h)$ in $B_{r_\varepsilon}(p)$ and  $\mathcal V_\varepsilon=\mathcal B_{r_\varepsilon}(\rho,\varphi)$ in $M\backslash B_{r_\varepsilon}(p)$ is a positive smooth solutions in $M\backslash \{p\}$ to the system \eqref{S} and satisfies the conditions \eqref{cond001}, \eqref{cond002} and \eqref{cond003} in the theorem.
\end{proof}

\section{High dimension: \texorpdfstring{$n\geq 6$}{Lg}}\label{sec011}

In this section we explain briefly how we can use the Weyl assumption to proof the main theorem for dimension $n\geq 6$. In the spirit of \cite{almir} this assumption is used in the interior analysis to assure that we can reduce the problem to a fixed point problem. In fact, in these dimensions the right hand side of \eqref{2.4} does not have the right decay to belong to the domain of the right inverse construct in Section \ref{sec008}. We have to be able to prove a similar lemma as Lemma \ref{lem004}. Since the main difference is in the interior analysis \eqref{sec009}, then we will explain only this part in order to get Theorem \ref{teo002} for high dimensions. The gluing construction (Section \ref{sec005}) follows in a similar way. For more details we refer the reader to \cite{almir}.

As in Section \ref{sec009} we perform a conformal change of metric to get the metric $g$ of that section. In normal coordinates at $p$, using the assumption in the Weyl tensor, namely, 
$$\nabla^k W_g(p)=0,\;\;k=1,\ldots,\left[\frac{n-6}{2}\right],$$
we get that $R_g=O(|x|^{\frac{n-4}{2}})$. Moreover, the it can be expand as $R_g=f+O(|x|^{n-3})$, where $f=O(|x|^{\frac{n-4}{2}})$ belongs to the high frequency space in each $\partial B_r(p)$. This decay it is not enough to the term $R_gu_{\varepsilon,R,a}$ of the right hand side of \eqref{2.4} to belong to the domain of the right inverse. To overcome this difficult an auxiliary function $w_{\varepsilon,R}$ is introduced. The function  $\mathcal W_{\varepsilon,R}\in C_{\mu_n;d}
^{2,\alpha}(B_r(0)\backslash\{0\})$, where $\mu_n=1$ for $n$ even and $1/2$ for $n$ odd, is such that $\mathcal L_{\varepsilon,R}(\mathcal W_{\varepsilon,R})=c_nfu_{\varepsilon,R}$, which the existence is given by Proposition \ref{propo003}. Therefore, instead we consider \eqref{SS} we consider $H_{g_0}(\beta^{-1}(\mathcal{U}_{\varepsilon,R,a} +\mathcal W_{\varepsilon,R}+h+\mathcal V_\phi+ \mathcal{V})) = 0$. In this way the bad term in \ref{2.4} disappear and we can apply the right inverse obtained in Corollary \ref{inver} to get the correspondent Theorem \ref{teo002} for high dimensions.




%
%

\end{document}